\documentclass[a4paper, 10.5 pt, reqno]{amsart}

\usepackage[english]{babel}
\usepackage{fullpage}
\usepackage{hyperref}
 \hypersetup{
     colorlinks=true,
     linkcolor=purple,
     filecolor=cyan,
     citecolor = cyan,      
     urlcolor=cyan,
     }
\usepackage[poorman]{cleveref}

\usepackage{amscd}
\usepackage[arrow, matrix, curve]{xy}
\usepackage{color}
\usepackage{blindtext}
\usepackage{tikz}
\usepackage{tikz-cd}
\usepackage{comment}

\usepackage{relsize} 
\usepackage[bbgreekl]{mathbbol} 
\usepackage{amsfonts} 
\usepackage[T1]{fontenc}
\DeclareSymbolFontAlphabet{\mathbb}{AMSb} 
\DeclareSymbolFontAlphabet{\mathbbl}{bbold}

\usepackage{mathrsfs}
\usepackage{aliascnt}

\newtheorem{Theorem}{Theorem}[section]

\newaliascnt{Lemma}{Theorem}
\newtheorem{Lemma}[Lemma]{Lemma}
\aliascntresetthe{Lemma}

\newaliascnt{Corollary}{Theorem}
\newtheorem{Corollary}[Corollary]{Corollary}
\aliascntresetthe{Corollary}

\newaliascnt{proposition}{Theorem}
\newtheorem{proposition}[proposition]{Proposition}
\aliascntresetthe{proposition}

\theoremstyle{definition}

\newaliascnt{Definition}{Theorem}
\newtheorem{Definition}[Definition]{Definition}
\aliascntresetthe{Definition}

\newaliascnt{Remark}{Theorem}
\newtheorem{Remark}[Remark]{Remark}
\aliascntresetthe{Remark}

\newaliascnt{Situation}{Theorem}
\newtheorem{Situation}[Situation]{Situation}
\aliascntresetthe{Situation}

\newaliascnt{Construction}{Theorem}
\newtheorem{Construction}[Construction]{Construction}
\aliascntresetthe{Construction}

\Crefname{Theorem}{Theorem}{Theorems}
\Crefname{Lemma}{Lemma}{Lemmas}
\Crefname{Corollary}{Corollary}{Corollaries}
\Crefname{Definition}{Definition}{Definitions}
\Crefname{Remark}{Remark}{Remarks}
\Crefname{Example}{Example}{Examples}
\Crefname{proposition}{Proposition}{Propositions}
\Crefname{construction}{Construction}{Constructions}
\Crefname{hope}{Hope}{Hopes}
\Crefname{Fact}{Fact}{Facts}
\Crefname{Conjecture}{Conjecture}{Conjectures}
\Crefname{Construction}{Construction}{Constructions}
\Crefname{Exercise}{Exercise}{Exercises}
\Crefname{Situation}{Situation}{Situations}
\Crefname{Warning}{Warning}{Warnings}
\Crefname{question}{Question}{Questions}

\usepackage{amssymb}

\DeclareMathOperator{\Hom}{Hom}

\DeclareMathOperator{\Gal}{Gal}

\DeclareMathOperator{\Spec}{Spec}
\DeclareMathOperator{\Spa}{Spa}

\DeclareMathOperator{\GL}{GL}

\DeclareMathOperator{\ad}{ad}

\DeclareMathOperator{\Z}{\mathbb{Z}}

\DeclareMathOperator{\perf}{perf}

\DeclareMathOperator{\Res}{Res}

\DeclareMathOperator{\Mor}{Mor}

\DeclareMathOperator{\op}{op}
\DeclareMathOperator{\Gr}{Gr}

\DeclareMathOperator{\Perf}{Perf}

\DeclareMathOperator{\Frob}{Frob}

\DeclareMathOperator{\coker}{coker}

\DeclareMathOperator{\tors}{tors}

\DeclareMathOperator{\sph}{sph}

\DeclareMathOperator{\sep}{sep}

\DeclareMathOperator{\ur}{ur}

\DeclareMathOperator{\Cent}{Cent}

\newcommand{\mbC}{\mathbb{C}}

\newcommand{\mbF}{\mathbb{F}}
\newcommand{\mbG}{\mathbb{G}}

\newcommand{\mbN}{\mathbb{N}}

\newcommand{\mbR}{\mathbb{R}}

\newcommand{\mbZ}{\mathbb{Z}}

\newcommand{\mcA}{\mathcal{A}}
\newcommand{\mcB}{\mathcal{B}}
\newcommand{\mcC}{\mathcal{C}}
\newcommand{\mcD}{\mathcal{D}}
\newcommand{\mcE}{\mathcal{E}}
\newcommand{\mcF}{\mathcal{F}}
\newcommand{\mcG}{\mathcal{G}}
\newcommand{\mcH}{\mathcal{H}}
\newcommand{\mcI}{\mathcal{I}}

\newcommand{\mcL}{\mathcal{L}}

\newcommand{\mcO}{\mathcal{O}}
\newcommand{\mcP}{\mathcal{P}}

\newcommand{\mcS}{\mathcal{S}}
\newcommand{\mcT}{\mathcal{T}}

\newcommand{\mfS}{\mathfrak{S}}

\newcommand{\id}{\mathrm{id}}
\def\Frob{\mathop{\mathrm{Frob}}\nolimits}

\newcommand{\lr}{\longrightarrow}

\newcommand{\simlr}{\overset{\sim}{\lr}}

\numberwithin{equation}{section}

\title{Close fields, affine Springer fibers and fundamental lemmas}
\author{Sebastian Bartling}
\address{Universität Duisburg-Essen, Fakultät für Mathematik, Thea-Leymann Straße, 45127 Essen, Germany}
\email{sebastian-bartling@hotmail.de}

\author{Kazuhiro Ito}
\address{Mathematical Institute, Tohoku University, 6-3, Aoba, Aramaki, Aoba-Ku, Sendai 980-8578, Japan}
\email{kazuhiro.ito.c3@tohoku.ac.jp}


\date{\today}
\begin{document}

\begin{abstract}
    We prove a geometric local constancy theorem for affine Springer fibers in families of close local fields. Consequently, stable orbital integrals are locally constant in these families, and both the base change fundamental lemma and the standard endoscopic fundamental lemma transfer from characteristic zero to arbitrary positive characteristic.
\end{abstract}
\maketitle

\section{Introduction}

In this paper, we introduce a simple geometric method based on the idea of close local fields for deducing certain identities between stable orbital integrals (or twisted $\kappa$-orbital integrals) in positive characteristic once they are known in characteristic zero. 
Our strategy relies on constructing affine Springer fibers over a family of close local fields.
This construction is inspired by the work of Li-Huerta \cite{LH}; however, 
our method is algebro-geometric in nature, in contrast to his analytic methods.
As an illustration, we provide an alternative proof of the group version of the standard endoscopic fundamental lemma for local function fields.
This group version has already been proved by
Casselman--Cely--Hales \cite{CasselmanCelyHales}
if the characteristic $p$ is large
and by Wang \cite{GriffinWangBook}
if $p$ is larger than twice the Coxeter number of the unramified reductive group $G$ under consideration.
However, to the best of our knowledge, the case of small characteristic $p$ has remained open.
Our method allows us to reduce the statement to the characteristic zero case, thereby establishing the fundamental lemma in arbitrary positive characteristic.
Furthermore, we prove the base change fundamental lemma for central elements in the parahoric Hecke algebras for general unramified reductive groups over local function fields, extending results of Ngô \cite{NgoBCFLGLn}, Lau \cite{EikePhD}, Henniart--Lemaire \cite{HenniartLemaireGLnBook}, Ray-Dulany \cite{Ray-Dulany} and Feng \cite{FengBCFLGLnParahoric}.
We remark that this method also applies to the arithmetic fundamental lemma of W.\ Zhang; this is the subject of the forthcoming paper \cite{AndreasMeAFL} of the first author with A.\ Mihatsch.

 \subsection{Main results and explanation of the method}
 The idea of close local fields goes back to the work of Krasner \cite{KrasnerCloseFields}, Deligne \cite{Deligne84} and Kazhdan \cite{KazhdanCloseField}.
 Roughly speaking, as the ramification of characteristic zero local fields grows to infinity, their theory approximates the theory of local function fields.
 Let $F_\infty=k((\pi_\infty))$ be a local function field with finite residue field $k$.
 We say that $\lbrace (F_{i},\pi_{i})\rbrace_{i\in \mathbb{N}\cup \lbrace \infty \rbrace}$ is a \textit{family of close local fields approximating} $F_{\infty}$ if $F_i$ is a non-archimedean local field of characteristic zero 
 with the same residue field $k$
 for $i\in \mathbb{N}$ and its absolute ramification index $e_{i}$ is greater than or equal to $i$.
 We have fixed uniformizers $\pi_{i}\in F_{i}$; many constructions given below will depend on the choice of $\pi_{i}$.

Let $G_\infty$ be an unramified reductive group over $F_\infty$.
Since $G_\infty$ is unramified, it arises as the base change of a reductive group scheme $\mcG$ over $W(k)$.
Let $G_i := \mcG \times_{\Spec W(k)} \Spec F_i$.
 The method of deducing identities between stable orbital integrals in the local function field case starts with a strongly regular semisimple element $\gamma_{\infty}\in G_{\infty}(F_{\infty})$ and spreads it out to a ``continuous'' family of strongly regular semisimple elements $\lbrace \gamma_{i}\in G_{i}(F_{i})\rbrace_{i\in \mathbb{N}_{\geq B}\cup \lbrace \infty \rbrace}$.
 By the Satake isomorphism, one has isomorphisms of spherical Hecke algebras $\mcH_{\sph}(G_{\infty}(F_{\infty}))\simeq \mcH_{\sph}(G_{i}(F_{i})),$ $f_{\infty}\mapsto f_i$.
 Our goal is to show that for stable orbital integrals we have
 \begin{equation}\label{eq: local constancy of stable orbital integrals intro}
 SO_{\gamma_{i}}(f_{i})=SO_{\gamma_{\infty}}(f_{\infty})
 \end{equation}
 for all $i\geq B,$ for some bound $B.$

Our approach to this problem relies on the well-known interpretation, due to Kazhdan--Lusztig \cite{KazhdanLusztig}, of stable orbital integrals in terms of counting points on (quotients of) affine Springer fibers. The construction of affine Springer fibers was recently extended to mixed characteristic by Chi \cite{ChiWittAffineSpringer}.
 Here we want to construct profinite families of affine Grassmannians and affine Springer fibers indexed by $i\in \mathbb{N}\cup \lbrace \infty \rbrace,$ the one-point compactification of the natural numbers.
The key idea is that once (quotients of) affine Springer fibers form a \textit{finitely presented} family over $\mathbb{N}\cup \lbrace \infty \rbrace$, point-counting becomes locally constant around $\infty$.
This geometric finiteness replaces harmonic analysis arguments.
 
To implement this idea, we make essential use of a construction that appeared for the first time in the work of Li-Huerta \cite{LH}.
Let $\psi_{i}\colon \mcO_{\infty}/\pi_{\infty}^{e_{i}}\simeq \mcO_{i}/\pi_{i}^{e_{i}}$
 be the $k$-algebra isomorphism
 sending $\pi_{\infty}$ to $\pi_{i}$, where $\mcO_i$ $(i \in \mbN \cup \{ \infty \})$ is the ring of integers of $F_i$.

 \begin{Construction}[Li-Huerta]
Let $\pi:=(\pi_{1}, \pi_{2},\dotsc,\pi_{\infty})\in \prod_{i\in \mathbb{N}}\mcO_{i} \times \mcO_{\infty}.$
Let
\[
\mcO \subset \prod_{i\in \mathbb{N}}\mcO_{i} \times \mcO_{\infty}
\]
be the $\pi$-adically complete subring
containing $\pi$
such that
$\mcO/\pi^{n}$ coincides with the subring of
$\prod_{i\in\mathbb{N}}\mcO_{i}/\pi_{i}^{n}\times \mcO_{\infty}/\pi_{\infty}^{n}$
consisting of elements
$x=(x_{1},x_{2}, \dotsc ,x_{\infty})$
with the property that there exists an integer $B \geq n$
such that for any $i \geq B$, we have $\psi_i(x_\infty)=x_i$.
We then define $F:=\mcO[1/\pi]$.
 \end{Construction}

One can spread out $G_\infty$ to a reductive group scheme over $\Spec F$, just by putting $G := \mcG \times_{\Spec W(k)} \Spec F$.
In fact, one can spread out an arbitrary parahoric group scheme $\mcP_{\infty}$ for $G_{\infty}$ to a smooth affine group scheme $\mcP$ over $\Spec \mcO,$ whose fibers are parahoric group schemes $\mcP_{i}$ of $G_{i}$; see Section \ref{Subsection:Unramified reductive groups in families}.
By considering such parahoric groups, we can establish the analogous local constancy \eqref{eq: local constancy of stable orbital integrals intro} more generally for parahoric Hecke algebras.

For a perfect $\mcO/\pi$-algebra $R,$ we define 
 $W_{\mcO}(R):=W(R)\widehat{\otimes}_{W(\mcO/\pi)}\mcO$
 (where we consider the $\pi$-adic completion)
 which is an $\mcO$-version of the ring of Witt vectors.
 We then define, as usual, the positive loop group (resp.\ loop group) functor on the category of perfect $\mcO/\pi$-algebras by $L^{+}_{\mcO}\mcP(R)=\mcP(W_{\mcO}(R))$ (resp.\ $L_{\mcO}G(R)=G(W_{\mcO}(R)[1/\pi])$).
The \textit{$\mcO$-Witt vector affine Grassmannian}
of $\mcP$ is the quotient sheaf $$\Gr_{\mcP}:=L_{\mcO}G/L^{+}_{\mcO}\mcP$$
 with respect to the \'etale topology.
 This lives over $\Spec \mcO/\pi$, whose topological space can be identified with
the profinite set
 $\mathbb{N}\cup \lbrace \infty \rbrace$.
 The fiber over $i \in \mbN$ is the Witt vector affine Grassmannian of $\mcP_{i}$, while the fiber over $\infty$ is the perfection of the usual affine Grassmannian of $\mcP_\infty$ defined in terms of power series.
 Our first main result is the following representability theorem.
 
 \begin{Theorem}[Corollary \ref{Corollary:affine Grassmannian of parahoric is representable}]
    The sheaf $\Gr_{\mcP}$ is representable by an increasing union of perfections of finitely presented projective schemes over $\Spec \mcO/\pi.$
 \end{Theorem}

Next, we state our results concerning affine Springer fibers. 
Let $\gamma_{\infty}\in G_\infty(F_{\infty})$ be a strongly regular semisimple element. 
After excluding finitely many $F_i$ ($i \in \mathbb{N}$) from $\{ (F_i, \pi_i) \}_{i \in \mbN \cup \{ \infty \}}$,
we can lift $\gamma_{\infty}$ to an element $\gamma \in G(F)$ 
whose fibers $\gamma_i \in G_i(F_i)$ are strongly regular semisimple elements. 
Consequently, the centralizer $G_{\gamma} \subset G$ is a maximal $F$-torus. 
In this setting, we can construct the affine Springer fibers $X_{\mcP, \gamma}^{w}\subset \Gr_{\mcP}$ for elements $w$ in the Iwahori--Weyl group (Definition \ref{def: Affine Springer fiber O}). 
These are perfect schemes, locally perfectly of finite presentation over $\Spec \mcO/\pi,$ which admit an action of $L_{\mcO} G_{\gamma}$.
The critical finiteness property of these objects is established in the following result.

\begin{Theorem}[Theorem \ref{Theorem: ASF is quasi-compact separated pfp algebraic space}]
    There exists an \'etale group scheme
    $\underline{\Lambda}_{\gamma} \subset L_{\mcO} G_{\gamma}$ such that the quotient
    $
    \underline{\Lambda}_{\gamma}\backslash X_{\mcP, \gamma}^{w}
    $
    is a (quasi-compact) perfect algebraic space perfectly of finite presentation over $\Spec \mcO/\pi.$
\end{Theorem}

This, in particular, implies that there exists a bound $B,$ such that for all $i\geq B,$
$$
\underline{\Lambda}_{\gamma_{i}}\backslash X_{\mcP_{i}, \gamma_i}^{w}\simeq \underline{\Lambda}_{\gamma_{\infty}}\backslash X_{\mcP_{\infty}, \gamma_{\infty}}^{w},
$$
where the left-hand side is the affine Springer fiber defined in terms of Witt vectors \cite{ChiWittAffineSpringer} and the right-hand side is the perfection of the affine Springer fiber defined in terms of power series. 
As alluded to above, stable orbital integrals are essentially 
given by counting $k$-valued points on $\underline{\Lambda}_{\gamma_{i}}\backslash X_{\mcP_{i}, \gamma_i}^{w}$; see \Cref{Subsection: Geometric interpretations of twisted kappa-orbital integrals}.
This then implies the desired equality \eqref{eq: local constancy of stable orbital integrals intro}
(also for parahoric Hecke algebras). 
In fact, we show a similar local constancy result for 
twisted $\kappa$-orbital integrals; see \Cref{Theorem: local constancy of kappa orbital integral}.
For this, we are led to consider a twisted version of affine Springer fibers, which does not seem to have appeared in the literature to the best of our knowledge; see \Cref{Subsection:O-Witt vector twisted affine Springer fibers}.

\subsection{Applications}
Let us state the two applications that we give in this paper; these were already mentioned above.
First, we discuss the base change fundamental lemma.
We give a quick formulation; for any unexplained term, we refer to \Cref{Section: BCFL}.
Let $F$ be a non-archimedean local field, $\widetilde{F}/F$ a degree $r$ unramified extension, $\sigma\in \Gal(\widetilde{F}/F)$ the Frobenius and $G$ an unramified reductive group over $F$.
Let $J=\mcP(\mcO_{F})\subset G(F)$
be a parahoric subgroup and let $\widetilde{J}=\mcP(\mcO_{\widetilde{F}}) \subset G(\widetilde{F})$
be the corresponding parahoric subgroup of $G(\widetilde{F})$.
We consider the parahoric Hecke algebras $\mcH_{J}(G(F))$
and $\mcH_{\widetilde{J}}(G(\widetilde{F}))$.
There exists a base change homomorphism 
$$
b\colon Z(\mcH_{\widetilde{J}}(G(\widetilde{F})))\rightarrow Z(\mcH_{J}(G(F)))
$$
defined by Haines using the Bernstein isomorphism.
Let $\delta\in G(\widetilde{F})$ and consider the norm $N(\delta)=\delta\sigma(\delta)\cdots \sigma^{r-1}(\delta)\in G(\widetilde{F}).$
Assume that there is a strongly regular semisimple element $\gamma\in G(F)$ such that $\gamma$ and $N(\delta)$ are $G(F^{\sep})$-conjugate.
Let $f\in Z(\mcH_{\widetilde{J}}(G(\widetilde{F})))$ and consider the stable twisted orbital integral $STO_{\delta}(f).$

\begin{Theorem}[Base change fundamental lemma, Theorem \ref{Thm: base change fundamental lemma}]
    The following equality holds for any non-archimedean local field $F$ of positive characteristic:
    \[
        STO_{\delta}(f)=SO_{\gamma}(b(f)).
    \]
\end{Theorem}

 The corresponding statement is known in characteristic zero for spherical Hecke algebras by the work of Clozel \cite{ClozelBCFL} and Labesse \cite{LabesseBCFL}, and this was extended by Haines \cite{HainesBCFLparahoric} to central elements in parahoric Hecke algebras. 
 Their proof strategies use global trace formula methods.
 The results in positive characteristic that we are aware of are due to Ngô \cite{NgoBCFLGLn}, Lau \cite{EikePhD}, Henniart--Lemaire \cite{HenniartLemaireGLnBook}, Ray-Dulany \cite{Ray-Dulany} and Feng \cite{FengBCFLGLnParahoric}, and they are all limited to the case of $\GL_{n}$.\footnote{T.\ Haines kindly informed us that Weimin Jiang, in his thesis, proves an analogous result for the Bernstein centers of depth zero principal series blocks for unramified reductive groups over local function fields by using trace formula methods. In particular, he gives an alternative proof of the base change fundamental lemma for Iwahori--Hecke algebras in positive characteristic (which differs from our own).}

As a second application, we discuss the standard endoscopic fundamental lemma in Section \ref{Section: Close fields and Endoscopy}.
Let $G$ be an unramified reductive group over $F$
and 
let $(H,s,\xi)$ be an unramified endoscopic datum for $G$ (\Cref{Definition: unramified endoscopic datum}).
Let
$\gamma_{H}\in H(F)$ and $\gamma\in G(F)$ be matching strongly regular semisimple elements, and let $\Delta_{0}(\gamma_{H},\gamma)$ be the Langlands--Shelstad transfer factor.
Consider $f \in \mcH_{\sph}(G(F))$ and let $b(f)\in \mcH_{\sph}(H(F))$ denote its transfer.

\begin{Theorem}[Standard endoscopic fundamental lemma, Theorem \ref{Theorem: Standard Endo FL}]
    The following equality holds for any non-archimedean local field $F$ of positive characteristic:
    \[
        \Delta_{0}(\gamma_{H},\gamma)O^{\kappa}_{\gamma}(f)=SO_{\gamma_{H}}(b(f)).
        \]
\end{Theorem}

As noted above, in contrast to \cite{CasselmanCelyHales} and \cite{GriffinWangBook}, our result is established without any characteristic restrictions, but we rely on the full strength of previous work by Ngô \cite{NgoFL}, Waldspurger \cite{WaldspurgerEndoscopyChangeofChar, Waldspurgertorduenestpassitordue} and Hales \cite{HalesReductionToUnit}.

\subsection{Relation to previous work and possible further applications}
The phenomenon of local constancy of orbital integrals in families of close local fields has already been observed by Lemaire in \cite{LemaireFourier96}.
His results concern $\GL_{n}$; furthermore, his methods are rooted in harmonic analysis and quite different from ours.
In a paper of Henniart--Lemaire \cite{HenniartLemaireInductionAutomorphe}, this approach is generalized to $\kappa$-orbital integrals and they deduce the fundamental lemma for the automorphic induction for $\GL_{n}$ in positive characteristic from the known characteristic zero case using the technique of close local fields.
Our results could also be applied to the fundamental lemma for automorphic induction, recovering their results.

Let us mention some possible further applications. First of all, it should be possible to extend our proof over local function fields of the base change fundamental lemma to the (weighted) fundamental lemma in twisted endoscopy.
Furthermore, our techniques also allow one to transfer the group version of the Jacquet--Rallis fundamental lemma from characteristic zero to arbitrary positive characteristic (cf.\ \cite{AndreasMeAFL}), thereby recovering the recent result of Wang--Zhang \cite{wang2026relativefundamentallemmasspherical} obtained under certain characteristic restrictions.
In another direction, it would be very interesting to see whether one can approach \cite[Conjecture 10.7]{HamacherKimPointCount} using our methods.
We hope to return to this question in the future.

\subsection{Structure of the article}
In Section \ref{Section:Reductive groups in families of close fields}, we provide some foundational material for families of close local fields and reductive groups in that setting. Then in Section \ref{Section:Affine Grassmannian in families of close fields} we construct the profinite family of affine Grassmannians and in Section \ref{Section: Affine Springer fibers in families of close fields} the construction of affine Springer fibers is given in that setting. In Section \ref{Section: Local constancy of stable twisted orbital integrals}, the local constancy result for twisted $\kappa$-orbital integrals is proved. After this we give in Section \ref{Section: BCFL} the proof of the base change fundamental lemma over local function fields, and in the final Section \ref{Section: Close fields and Endoscopy}, our applications to the standard endoscopic fundamental lemma over local function fields are discussed.

\section{Reductive groups in families of close fields}\label{Section:Reductive groups in families of close fields}

In this section, for a family of close fields $\{ (F_i, \pi_i) \}_{i \in \mbN \cup \{ \infty \}}$, we first recall the definition of the associated ring $F$ introduced in the work of Li-Huerta \cite{LH}, and study its fundamental properties. 
We then explain how several topics concerning reductive groups (e.g.\ cocharacter groups of $F$-tori and Bruhat--Tits group schemes) can be naturally developed over this ring $F$.

\subsection{Families of close fields}\label{Subsection:The ring O}

Let $k$ be a perfect field of characteristic $p >0$.
For $i\in \mathbb{N},$ let $F_{i}$ be a totally ramified extension of $W(k)[1/p]$ of degree $e_i$
with ring of integers $\mathcal{O}_{F_{i}}$ and fixed uniformizer $\pi_{i}$. 
Let $F_\infty := k(\!(\pi_\infty)\!)$ with ring of integers $\mcO_{F_\infty}=k[[\pi_\infty]]$.
For every $i \in \mbN$, we denote by
\[
	\psi_{i}\colon \mathcal{O}_{F_\infty}/\pi_{\infty}^{e_{i}} \simlr \mcO_{F_i}/\pi_i^{e_i}
\]
the isomorphism over $k$ sending $\pi_\infty$ to $\pi_i$. 
We assume that $e_{i}\geq i.$
We call such a family $F=\{ (F_i, \pi_i) \}_{i \in \mbN \cup \{ \infty \}}$ a \textit{family of close fields} (with residue field $k$).
We also say that
$F$
is a 
\textit{family of close fields approximating} $F_\infty$.
When $k$ is a finite field, we also call $F$ a \textit{family of close local fields}.

\begin{Definition}[Li-Huerta]\label{def:O}
Let $\pi:=(\pi_{1}, \pi_{2}, \dotsc ,\pi_{\infty})\in \prod_{i\in \mathbb{N}}\mcO_{F_i} \times \mcO_{F_\infty}.$
Let
\[
\mcO_F \subset \prod_{i\in \mathbb{N}}\mcO_{F_i} \times \mcO_{F_\infty}
\]
be the $\pi$-adically complete subring
containing $\pi$
such that
the quotient
$\mcO_F/\pi^{n}$ coincides with the subring of
$\prod_{i\in\mathbb{N}}\mcO_{F_i}/\pi_{i}^{n}\times \mcO_{F_\infty}/\pi_{\infty}^{n}$
consisting of elements
$x=(x_{1},x_{2}, \dotsc,x_{\infty})$
with the property that there exists an integer $B \geq n$
such that for any $i \geq B$, we have $\psi_i(x_\infty)=x_i$.
For any $B \geq 1$,
let $\mcO_{F, \geq B} \subset \prod_{i \geq B }\mcO_{F_i} \times \mcO_{F_\infty}$
be the image of $\mcO_F$ under the natural projection; then $\mcO_F = \prod_{1 \leq i \leq B-1} \mcO_{F_i} \times \mcO_{F, \geq B}$.
We write $F:=\mcO_F[1/\pi]$
and $F_{\geq B}:=\mcO_{F, \geq B}[1/\pi]$
(by abuse of notation).
We often drop $F$ from the notation when there is no confusion.
For example, we write
$\mathcal{O}_i=\mathcal{O}_{F_i}$, $\mathcal{O}=\mathcal{O}_{F}$ and $\mathcal{O}_{\geq B}=\mathcal{O}_{F, \geq B}$, and for a scheme $X$ over $F$, we write $X_i := X_{F_i}$ and $X_{\geq B} := X_{F_{\geq B}}$.
(Throughout this paper, a subscript usually indicates base change.)
\end{Definition}

\begin{Remark}\label{Remark:O mop pi as filtered rings}
    Let $B \geq n \geq 1$.
    Let $(\mathcal{O}_{F, [B, B'], n}, f_{B'})_{B' \geq B}$
    be the filtered system of discrete rings where
    \[
    \mathcal{O}_{F, [B, B'], n} = \prod^{B'}_{i=B}\mathcal{O}_{F_i}/\pi_{i}^{n} \times \mathcal{O}_{F_\infty}/\pi_{\infty}^{n}
    \]
    and $
	f_{B'}\colon \mathcal{O}_{F, [B, B'], n} \to \mathcal{O}_{F, [B, B'+1], n}
$
are the faithfully flat transition maps defined by
$
	(x_{B}, \dotsc,x_{B'},x_{\infty})\mapsto (x_{B},\dotsc,x_{B'}, \psi_{B'+1}(x_{\infty}),x_{\infty}).
$
It follows that
$$
	\varinjlim_{B'\geq B} \mathcal{O}_{F, [B, B'], n}=\mcO_{F, \geq B}/\pi^{n}.
$$
\end{Remark}

\begin{Remark}\label{Remark:adic space Spa E}
    By \cite[Proposition 2.5]{LH},
    the Huber pair $(F, \mcO_F)$ is sheafy, and the underlying topological space of the adic space $\Spa F:=\Spa (F, \mcO_F)$ is naturally homeomorphic to the one-point compactification $\mbN \cup \{ \infty \}$.
    By the proof of \cite[Proposition 2.8]{LH}, $\mcO_{\Spa F, \infty} = \varinjlim_{B} F_{\geq B}$ is a henselian local ring whose residue field is $F_\infty$.
\end{Remark}

\begin{Lemma}\label{Lemma: elements of O}
    Let $a \in F$ be an element such that $a_\infty \in \pi^n_\infty \mcO_{\infty} \backslash \pi^{n+1}_\infty \mcO_{\infty}$ for an integer $n$.
    Then, for some $B \geq 1$, the image $a_{\geq B} \in F_{\geq B}$ of $a$ is contained in $\pi^n \mcO_{\geq B}$ and 
    $a_i \in \pi^n_i \mcO_{i} \backslash \pi^{n+1}_i \mcO_{i}$
    for all $i \geq B$.
\end{Lemma}

\begin{proof}
    This follows immediately from the construction.
\end{proof}

\begin{Corollary}\label{Corollary: open nbd of infinity}
    Let $\infty \in \Spec F$ be the closed point corresponding to $F \to F_\infty$.
    The open subsets $\{ \Spec F_{\geq B} \}_{B \geq 1}$
    of $\Spec F$ form a fundamental system of open neighborhoods of $\infty$.
\end{Corollary}

\begin{proof}
    Let $a \in F$ be such that
    $\infty \in D(a) \subset \Spec F$, or equivalently
    $a_\infty \in F^\times_\infty$.
    We have to show that $\Spec F_{\geq B} \subset D(a)$ for some $B \geq 1$.
    By \Cref{Lemma: elements of O}, there exist
    integers $n$ and $B \geq 1$ such that $\pi^n a \in (\mcO_{\geq B})^{\times}$, and hence $a \in (F_{\geq B})^{\times}$.
    This means that $\Spec F_{\geq B} \subset D(a)$.
\end{proof}

Let $\overline{k}$ be an algebraic closure of $k$.
We denote by $F^{\ur}_i:=\varinjlim_{k'} F_i \otimes_{W(k)} W(k')$ the maximal unramified extension of $F_i$, where $k'$ runs over all finite extensions of $k$ in $\overline{k}$,
and by $\breve{F}_i$ the completion of $F^{\ur}_i$.
Then $\breve{F}= \{ (\breve{F}_i, \pi_i) \}_{i \in \mbN \cup \{ \infty \} }$ is a family of close fields with residue field $\overline{k}$, and thus we can form $\mcO_{\breve{F}}$ and $\breve{F}=\mcO_{\breve{F}}[1/\pi]$ as in \Cref{def:O}.
It will also be useful to consider
$\mcO_{F^{\ur}} := \varinjlim_{k'} \mcO_F \otimes_{W(k)} W(k')$
and $F^{\ur}:=\mcO_{F^{\ur}}[1/\pi]$.
We define $\mcO_{F^{\ur}, \geq B}$ and $F^{\ur}_{\geq B}$ in the same way.

For $i \in \mbN \cup \{ \infty \}$, let $\Gamma_i=\Gal(F^{\sep}_i/F_i)$ be the absolute Galois group of $F_i$ where $F^{\sep}_i$ is a separable closure of $F_i$.
We fix an embedding $F^{\ur}_i \hookrightarrow F^{\sep}_i$. 
For $-1 \leq n$, let
$I^n_i \subset \Gamma_i$
be the $n$-th ramification subgroup (as defined in \cite[Chapter IV, Section 3, Remark 1]{SerreLocalfields}).
By \cite{Deligne84},
given the choices of $\pi_i$ and $\pi_\infty$,
we have an isomorphism
\[
    \psi_i \colon \Gamma_{\infty}/I^{e_i}_{\infty}  \overset{\sim}{\to} \Gamma_i/I^{e_i}_i
\]
which is canonical up to inner automorphisms.
This induces 
$I^f_{\infty}/I^{e_i}_{\infty}
\simlr I^f_i/I^{e_i}_i$ for any $-1 \leq f \leq e_i$.
We may and do assume that $\psi_i$ induces the identity
$\Gamma_\infty/I^0_\infty = \Gal(\overline{k}/k) \simlr \Gamma_i/I^0_i = \Gal(\overline{k}/k)$.

We recall an explicit construction of finite \'etale algebras over $F$ given in \cite[Proposition 2.8]{LH}.

\begin{Construction}\label{Construction: etale covering of family of close fields}
    Let $E_\infty$ be a finite separable extension of $F_\infty$
    with residue field $k'$.
    We use the superscript $'$ to indicate base change to $W(k')$.
    For example, we write $F' := F \otimes_{W(k)} W(k')$
    and $\mcO' := \mcO \otimes_{W(k)} W(k')$.
    We choose a uniformizer $\pi_{E_\infty}$ of $E_\infty$
    and let
    \[
    f_\infty(X)= X^r + \pi_\infty( \sum_{1 \leq j \leq r}a_{\infty, j}X^{r-j}) \in \mcO'_\infty[X]
    \]
    be the Eisenstein polynomial of $\pi_{E_\infty}$ (so $a_{\infty, r} \in (\mcO'_\infty)^\times)$.
    We identify $F'_\infty[X]/(f_\infty(X))$ with $E_\infty$ by $X \mapsto \pi_{E_\infty}$.
For each $i \in \mbN$, we choose
    $a_{i, j} \in \mcO'_i$ ($1 \leq j \leq r$) such that $\psi_i(\overline{a}_{\infty, j}) = \overline{a}_{i, j}$ in $\mcO'_{i}/\pi_i^{e_{i}}$.
    Then
    \[
    f_i(X) := X^r + \pi_i( \sum_{1 \leq j \leq r}a_{i, j}X^{r-j}) \in \mcO'_i[X]
    \]
    is an Eisenstein polynomial, and $E_i:=F'_i[X]/(f_i(X))$ is a totally ramified extension of $F'_i$.
    Let $\pi_{E_i} \in E_i$ ($i \in \mbN$) be the image of $X$.
    Then $\{ (E_i, \pi_{E_i}) \}_{i \in \mbN \cup \{ \infty \} }$ forms a family of close fields.
    The ring $E$ associated with $\{ (E_i, \pi_{E_i}) \}_{i \in \mbN \cup \{ \infty \} }$ is a finite \'etale algebra over $F$.
    We note that 
    $\mcO_E$ can be identified with
    $\mcO'[X]/(X^r + \pi( \sum_{1 \leq j \leq r}a_{j}X^{r-j}))$
    where $a_j=(a_{1, j}, a_{2, j}, \dotsc, a_{\infty, j}) \in \mcO'$.
    
    We assume that $E_\infty$ is a finite Galois extension of $F_\infty$ and $I^n_\infty$ acts trivially on $E_\infty$ for some $n \geq 1$.
    It then follows that, for $i \geq n$, the finite extension $E_i$ of $F_i$ is Galois, and $I^n_i$ acts trivially on $E_i$.
    In fact, with the choice of $\pi_{E_i}$, we can identify the triples $\mathrm{Tr}_{rn}(E_i)$
    and $\mathrm{Tr}_{rn}(E_\infty)$ (in the sense of \cite{Deligne84}).
    Then by \cite[Th\'eor\`eme 2.8]{Deligne84}, we have a canonical isomorphism 
    \[
    \theta_i \colon \Gal(E_\infty/F_\infty) \simlr \Gal(E_i/F_i)
    \]
    for each $i \geq n$. 
    There exists an integer $B \geq 1$ such that $\Spec E_{\geq B} \to \Spec F_{\geq B}$ has the structure of a $\Gal(E_\infty/F_\infty)$-torsor
    such that the restriction of the action of an element $g_\infty \in \Gal(E_\infty/F_\infty)$ to the fiber at $i \geq B$ (resp.\ $i= \infty$) coincides with the action of $\theta_i(g_\infty)$ (resp.\ $g_\infty$).
    Moreover, for an $F'_\infty$-embedding $E_\infty \hookrightarrow F^{\sep}_\infty$, there is an $F'_i$-embedding
    $E_i \hookrightarrow F^{\sep}_i$ such that the following diagram commutes:
    \begin{equation}\label{equation:diagram of Galois group of close fields}
        \vcenter{\xymatrix{
\Gamma_\infty/I^n_\infty \ar^-{}[r]  \ar[d]^-{\psi_i}  &    \Gal(E_\infty/F_\infty) \ar[d]^-{\theta_i} \\
\Gamma_i/I^n_i \ar[r]^-{} & \Gal(E_i/F_i).
}}
    \end{equation}
\end{Construction}

For a perfect $\mcO/\pi$-algebra $R$, we define
\[
W_{\mcO, n}(R) := W(R)\otimes_{W(\mcO/\pi)} \mcO/\pi^n \quad \text{and} \quad W_\mcO(R):= \varprojlim_{n} W_{\mcO, n}(R).
\]
Here $W(\mcO/\pi)\rightarrow \mcO$
is the map corresponding to $\mcO/\pi \simeq \varprojlim_{x \mapsto x^p} \mcO/p$ by adjunction.

Let $\Perf_{\mcO/\pi}$ be the category of perfect $\mcO/\pi$-algebras.
The fibered category over $\Perf^{\op}_{\mcO/\pi}$ which associates to a perfect $\mcO/\pi$-algebra $R$ the category of finite projective modules over $W_\mcO(R)$ satisfies \'etale descent
(since for every $n$, $\Spec W_{\mcO, n}(R') \to \Spec W_{\mcO, n}(R)$ is \'etale and surjective if and only if so is $\Spec R' \to \Spec R$).
It follows that for an affine scheme $X$ over $\Spec \mcO$,
the presheaves
$L_{\mcO} X$ and $L^+_{\mcO} X$ defined by
$R \mapsto X(W_\mcO(R)[1/\pi])$ and $R \mapsto X(W_\mcO(R))$, respectively, are \'etale sheaves.

We also consider the sheaf
$L^n_{\mcO} X$ defined by $R \mapsto X(W_{\mcO, n}(R))$.
Recall that for each $i \in \mbN \cup \{ \infty \}$, if $X_i$ is a finite type $\mcO_i$-scheme, then by \cite{Greenberg61} the sheaf
$L^n_{\mcO_i} X_i$
defined by $R \mapsto X(W_{\mcO_i, n}(R))$
is representable by the perfection of a finite type $k$-scheme, which we also denote by $L^n_{\mcO_i} X_i$.
Here $W_{\mcO_i, n}(R):= W(R)\otimes_{W(k)} \mcO_i/\pi^n_i$.
The analogous result in our case is the following:

\begin{Lemma}\label{Lemma:representability of L^n X}
    Let $X$ be a finitely presented $\mcO$-scheme
    and let $X_i:= X \otimes_\mcO \mcO_i$.
    Then there is an integer $B \geq 1$ such that $L^n_{\mcO} X$ is isomorphic to the base change of
    $(\coprod_{i \leq B} L^n_{\mcO_i}X_i) \coprod L^n_{\mcO_\infty}X_\infty$
    over $\mcO_{[1, B], 1}= \prod_{i \leq B}\mathcal{O}_{i}/\pi_{i} \times \mathcal{O}_{\infty}/\pi_{\infty}$
    along the map $\mcO_{[1, B], 1} \to \mcO/\pi$.
    In particular, $L^n_{\mcO} X$ is representable by the perfection of a finitely presented $\mcO/\pi$-scheme.
\end{Lemma}

\begin{proof}
    Since $X \otimes_\mcO \mcO/\pi^n$ is of finite presentation over $\mcO/\pi^n= \varinjlim_{B} \mcO_{[1, B], n}$, there is an integer $B \geq 1$ such that $X \otimes_\mcO \mcO/\pi^n$ is isomorphic to the base change of a finite type $\mcO_{[1, B], n}$-scheme $Y$.
    Since $\mcO_{[1, B], n} \to \mcO/\pi^n$ has a splitting, we see that $Y$ is isomorphic to 
    $(\coprod_{i \leq B} X_i\otimes_{\mcO_i} \mcO_i/\pi^n_i) \coprod X_\infty\otimes_{\mcO_{\infty}} \mcO_\infty/\pi^n_\infty$
    over $\mcO_{[1, B], n}$, from which the assertion is clear.
\end{proof}

\subsection{Tori in families of close fields}\label{Tori in families of close fields}

For a torus $T$ over an affine scheme $\Spec A$, we denote by
$\Hom_A(\mbG_{m}, T)$
the group of cocharacters $\mbG_{m} \to T$ over $\Spec A$.

\begin{Situation}\label{Situation: torus}
    Let $T$ be a torus over $\Spec F$.
Let $E_\infty$ be a finite Galois extension of $F_\infty$ such that $T_\infty$ is split over $E_\infty$.
We extend $E_\infty$ to a family of close fields
$E= \{ (E_i, \pi_{E_i}) \}_{i \in \mbN \cup \{ \infty \}}$
as in \Cref{Construction: etale covering of family of close fields}.
Since $\mcO_{\Spa E, \infty}$ is henselian, $T_{\mcO_{\Spa E, \infty}}$ is also split.
Since $\mcO_{\Spa E, \infty} = \varinjlim_B E_{\geq B}$, we have
    \[
    \varinjlim_B \Hom_{E_{\geq B}}(\mbG_{m}, T_{E_{\geq B}}) \simeq \Hom_{\mcO_{\Spa E, \infty}}(\mbG_{m}, T_{\mcO_{\Spa E, \infty}}) \simeq \Hom_{E_\infty}(\mbG_{m}, T_{E_\infty}).
    \]
Since $\Hom_{E_\infty}(\mbG_{m}, T_{E_\infty})$ is a free abelian group of finite rank, there exists a splitting
\begin{equation}\label{equation:lattice in cocharacter group}
    \Hom_{E_\infty}(\mbG_{m}, T_{E_\infty}) \to \Hom_{E_{\geq B}}(\mbG_{m}, T_{E_{\geq B}})
\end{equation}
for some $B \geq 1$, and any two such splittings coincide with each other over $E_{\geq B'}$ for some $B' \geq B$.
We assume that $B$ is large enough such that $\Spec E_{\geq B} \to \Spec F_{\geq B}$ forms a $\Gal(E_\infty/F_\infty)$-torsor as in \Cref{Construction: etale covering of family of close fields}.
After enlarging $B$, we may also assume that the statement of the following proposition holds for the splitting \eqref{equation:lattice in cocharacter group}.
\end{Situation}

\begin{proposition}\label{Proposition: lattices in cocharacter groups}
After enlarging $B$, the splitting \eqref{equation:lattice in cocharacter group}
is
$\Gal(E_\infty/F_\infty)$-equivariant, and for any $i \geq B$, the composition
    \[
    \Hom_{E_\infty}(\mbG_{m}, T_{E_\infty}) \to \Hom_{E_{\geq B}}(\mbG_{m}, T_{E_{\geq B}}) \to \Hom_{E_i}(\mbG_{m}, T_{E_i})
    \]
    is bijective.
\end{proposition}

\begin{proof}
    The first assertion is clear because $\Gal(E_\infty/F_\infty)$ is finite.
    After enlarging $B$, we have an isomorphism
    $
    T_{E_{\geq B}} \simeq \mbG^N_{m}.
    $
    Via this isomorphism, we identify both $\Hom_{E_\infty}(\mbG_{m}, T_{E_\infty})$
    and $\Hom_{E_i}(\mbG_{m}, T_{E_i})$ for $i \geq B$ with $\mbZ^N$.
    It suffices to prove that the composition is the identity for all $i \geq B$ after enlarging $B$.
    Let $\lambda_j$ denote the cocharacter of $T_{E_{\geq B}}$ corresponding to the $j$-th component $\mbG_m$ ($1 \leq j \leq N$).
    After enlarging $B$, we may assume that the image of $(\lambda_j)_\infty$ under \eqref{equation:lattice in cocharacter group} is $\lambda_j$ for all $j$.
    Then the claim is clear.
\end{proof}

We set $X_*(T_i):=\Hom_{F^{\sep}_i}(\mbG_{m}, T_{F^{\sep}_i})$.
We fix an $F'_\infty$-embedding $E_\infty \hookrightarrow F^{\sep}_\infty$ and
    let $n \geq 1$ be an integer such that $I^n_\infty$ acts trivially on $E_\infty$.
    For each $i \geq n$, we fix an $F'_i$-embedding $E_i \hookrightarrow F^{\sep}_i$ which makes the diagram \eqref{equation:diagram of Galois group of close fields} commute.
    We assume that $B \geq n$.
    Then we can identify $X_*(T_i)$ with $\Hom_{E_i}(\mbG_{m}, T_{E_i})$ for any $n \leq i \leq \infty$, and the composition in \Cref{Proposition: lattices in cocharacter groups}
    induces an isomorphism
    \begin{equation}\label{equation: psi for cocharacter groups}
        \psi_i \colon X_*(T_\infty) \simeq X_*(T_i)
    \end{equation}
    for all $i \geq B$ that is equivariant with respect to the actions of 
    $\Gamma_\infty/I^{n}_\infty \simeq \Gamma_{i}/I^{n}_{i}$.

\begin{Corollary}\label{Corollary: local constancy of cocharacter groups}
    Keep the notation as above.
    Let $t \in T(F^{\ur}_{\geq B})$.
    There exists an integer $B' \geq B$ such that for any $i \geq B'$, we have $\psi_i(w_{T_\infty}(t_\infty))=w_{T_i}(t_i)$ in $X_*(T_i)_{I_i}$.
    Here $w_{T_i} \colon T_i(F^{\ur}_i) \to X_*(T_i)_{I_i}$ is the Kottwitz homomorphism, and we denote the induced isomorphism
    $X_*(T_\infty)_{I_\infty} \simlr X_*(T_i)_{I_i}$ by the same symbol $\psi_i$.
\end{Corollary}
\begin{proof}
After enlarging $B$,
    we see that
    $\Spec E^{\ur}_{\geq B} \to \Spec F^{\ur}_{\geq B}$
    is naturally a $\Gal(E_\infty/F'_\infty)$-torsor.
    We consider the norm map
    \[
    \mathrm{Nm} \colon \Res^{E^{\ur}_{\geq B}}_{F^{\ur}_{\geq B}} T_{E^{\ur}_{\geq B}} \to T_{F^{\ur}_{\geq B}}.
    \]
    We first prove that after enlarging $B$, there is an element $\tilde{t} \in T(E^{\ur}_{\geq B})$
    such that $\mathrm{Nm}(\tilde{t})=t$.
    Let $\mcO_{F^{\ur}, \infty} := \varinjlim_{B'} 
    F^{\ur}_{\geq B'}$
    and 
    $\mcO_{E^{\ur}, \infty} := \varinjlim_{B'} 
    E^{\ur}_{\geq B'}$.
    It suffices to prove that $t \in T(\mcO_{F^{\ur}, \infty})$ has a lift
    in $T(\mcO_{E^{\ur}, \infty})$ along $\mathrm{Nm}$.
    (We use the same notation for the image of $t$ in $T(\mcO_{F^{\ur}, \infty})$.)
    By Steinberg's theorem, $t_\infty \in T(F^{\ur}_\infty)$ has a lift $\tilde{t}_\infty \in T(E^{\ur}_\infty)$.
    Since $\mcO_{F^{\ur}, \infty}$ is a henselian local ring with residue field $F^{\ur}_\infty$ and $\mathrm{Nm}$ is smooth, it follows that 
    $t$ has a lift in $T(\mcO_{E^{\ur}, \infty})$.

    As in the proof of \Cref{Proposition: lattices in cocharacter groups}, we may assume that
    $
    T_{E_{\geq B}} \simeq \mbG^N_{m}.
    $
    Let $\lambda_j$ (resp.\ $\mu_j$)
    denote the cocharacter
    (resp.\ the character) of $T_{E_{\geq B}}$ corresponding to the $j$-th component $\mbG_m$.
    Then $w_{T_i}(t_i) \in X_*(T_i)_{I_i}$
    is the image of $\sum^N_{j=1} \omega_i(\mu_j(\tilde{t}_i))\lambda_j \in X_*(T_i)$, where $\omega_i \colon (E^{\ur}_i)^\times \twoheadrightarrow \mbZ$ is the (normalized) valuation.
    Therefore, the desired assertion follows from \Cref{Lemma: elements of O}.
\end{proof}

If $k$ is a finite field,
then we have
$H^1(F_i, T_i) \simeq X_*(T_i)_{\Gamma_i, \mathrm{tor}}$
by Tate--Nakayama duality.
We need the following compatibility between Tate--Nakayama duality and families of close local
fields.

\begin{Corollary}\label{Corollary:Tate-Nakayama compatibility}
Assume that $k$ is finite.
    Let $c \in H^1(\Gal(E_\infty/F_\infty), T(E_{\geq B}))$, and for $B \leq i \leq \infty$, let $c_i \in H^1(F_i, T_i)$ be the image of $c$.
    Via Tate--Nakayama duality, we regard $c_i$ as an element of $X_*(T_i)_{\Gamma_i, \mathrm{tor}}$.
    Then there exists an integer $B' \geq B$ such that for all $i \geq B'$,
    we have $\psi_i(c_\infty)=c_i$, where we denote the induced isomorphism
    $X_*(T_\infty)_{\Gamma_\infty, \mathrm{tor}} \simlr X_*(T_i)_{\Gamma_i, \mathrm{tor}}$ by the same symbol $\psi_i$.
\end{Corollary}

\begin{proof}
    Let $\mcE$ be a $T$-torsor over $\Spec F_{\geq B}$ corresponding to $c$.
    By Steinberg's theorem, we may assume that $\mcE \times_{\Spec F_{\geq B}} \Spec F'_\infty$ is a trivial $T$-torsor.
    Since $T$ is smooth, then we may assume after enlarging $B$ that
    $\mcE \times_{\Spec F_{\geq B}} \Spec F'_{\geq B}$ is trivial.
    It follows that $c$ is represented by the inflation of
    a $1$-cocycle
    $g \mapsto t_g$
    of $\Gal(F'_\infty/F_\infty)$ in $T(F'_{\geq B})$.
    Then $c_i \in X_*(T_i)_{\Gamma_i, \mathrm{tor}}$
    coincides with the image of $w_{T_i}(t_{\sigma, i}) \in X_*(T_i)_{I_i}$
    in $X_*(T_i)_{\Gamma_i}$, where $\sigma \in \Gal(F'_\infty/F_\infty)$ is the Frobenius element; see \cite[Section 2]{KottwitzIso1}.
    Therefore, the claim follows from \Cref{Corollary: local constancy of cocharacter groups}.
\end{proof}

We give a basic example of a torus over $F$. 
For a reductive group scheme $G$ over an affine scheme $\Spec A$ and a section $\gamma \in G(A)$, we denote by $G_\gamma$ the centralizer of $\gamma,$ i.e.\
the subgroup scheme such that 
$
G_\gamma(R)=\lbrace g\in G(R) \colon g^{-1}\gamma g = \gamma \rbrace
$
for $A$-algebras $R$.
Note that $G_\gamma$ is a finitely presented closed subgroup scheme of $G$.
We say that $\gamma \in G(A)$ is a \textit{fiberwise strongly regular semisimple element} if 
the fiber $\gamma_s \in G(s)$ is strongly regular semisimple (i.e.\ $(G_s)_{\gamma_s}$ is a maximal torus of $G_s$)
for every geometric point $s$ of $\Spec A$.

\begin{Lemma}\label{Lemma: centralizer of fiberwise strongly regular semisimple}
    Let $\gamma \in G(A)$ be a fiberwise strongly regular semisimple element.
    Then $G_\gamma \subset G$ is a maximal torus.
\end{Lemma}
    
\begin{proof}
    By \cite[Expos\'e XIII, Th\'eor\`eme 3.1]{SGA3II}, there is a unique maximal torus $T \subset G$ containing $\gamma$.
    The inclusion
    $T \hookrightarrow G_\gamma$
    is an isomorphism on fibers, and hence an isomorphism by \cite[Lemma B.3.1]{ConradReductiveGroupSchemes}.
\end{proof}

Now let $G$ be a reductive group scheme over $\Spec F$.

\begin{proposition}\label{Proposition: strongly regular-semi-simple spreads out}
Let $\gamma_\infty \in G(F_\infty)$ be an element.
\begin{enumerate}
    \item There exist $B \geq 1$ and a lift $\gamma \in G(F_{\geq B})$ of $\gamma_\infty$.
    \item Assume that $\gamma_\infty \in G(F_\infty)$ is strongly regular semisimple.
    Then any lift $\gamma \in G(F_{\geq B})$ of $\gamma_\infty$ becomes fiberwise strongly regular semisimple after enlarging $B$.
\end{enumerate}
\end{proposition}

\begin{proof}
(1) follows since $G$ is smooth,
the local ring $\mcO_{\Spa F, \infty}$ is henselian, and $\mcO_{\Spa F, \infty}= \varinjlim_{B \geq 1} F_{\geq B}$.
Since the strongly regular semisimple elements form an open subset of $G$, the assertion (2) immediately follows from \Cref{Corollary: open nbd of infinity}.
\end{proof}

\subsection{Unramified reductive groups and parahoric subgroups in families of close fields}\label{Subsection:Unramified reductive groups in families}

\begin{Situation}\label{Situation: unramified reductive group in families of close fields}
    Let $\mcG$ be a quasi-split reductive group scheme over $\Spec W(k)$ with
    a maximal torus $\mcT$ and a Borel subgroup $\mcB$ containing $\mcT$.
    We fix a $\Gal(\overline{k}/k)$-stable pinning
    $\{ x_\alpha \}_{\alpha \in \Delta}$
    of $(\mcG, \mcT, \mcB)$ over $W(k)^{\ur} = \varinjlim_{k'}  W(k')$.
    Let $F= \{ (F_i, \pi_{i}) \}_{i \in \mbN \cup \{ \infty \}}$ be a family of close fields with residue field $k$.
    We write
    $(G_F, T_F, B_F):=(\mcG_F, \mcT_F, \mcB_F)$
    and
    $(G_{F_i}, T_{F_i}, B_{F_i}):=(\mcG_{F_i}, \mcT_{F_i}, \mcB_{F_i})$
    for each $i \in \mbN \cup \{ \infty \}$.
    If there is no ambiguity, we simply write
    \[
    (G, T, B):=(G_F, T_F, B_F)
    \quad \text{and} \quad (G_i, T_i, B_i):=    (G_{F_i}, T_{F_i}, B_{F_i}).
    \]
    We note that every unramified reductive group  over $F_i$ $(i \in \mbN \cup \{ \infty \})$ naturally arises in this way.
    Let $\mcS \subset \mcT$ be the maximal split torus over $W(k)$.
    Then $S_i:= \mcS_{F_i}$ is the maximal split subtorus of $T_i$.
    Let $\mcG_{\mathrm{der}} \subset \mcG$ be the derived subgroup scheme and let $\mcG_{\mathrm{sc}} \to \mcG_{\mathrm{der}}$ be the simply connected cover.
    Let $\mcS_{\mathrm{sc}} \subset \mcG_\mathrm{sc}$ be the preimage of $\mcS \subset \mcG$
and write $S_{\mathrm{sc}, i} := (\mcS_{\mathrm{sc}})_{F_i}$.
\end{Situation}

 We write $\mfS:=W(k)[[t]]$.
    Then we define
    \[
    \widetilde{W}:=\widetilde{W}(G_F):=N_{\mcG}(\mcS)(\mfS[1/t])/\mcT(\mfS)
    \]
    which we call the \textit{Iwahori--Weyl group} of $G_F$.
The following lemma justifies our terminology.

\begin{Lemma}\label{Lemma:Iwahori Weyl group}
Let $i \in \mbN \cup \{ \infty \}$.
Let $\widetilde{W}_i:=N_{G_i}(S_i)(F_i)/\mcT_i(\mcO_{i})$ be the Iwahori--Weyl group of $G_i$.
The map $\mfS \to \mcO_{i}$ defined by $t \mapsto \pi_i$ induces an isomorphism
$
\widetilde{W} \simlr \widetilde{W}_i.
$
\end{Lemma}

\begin{proof}
    Note that the quotient $N_{\mcG}(\mcS)/\mcT$ is the constant group scheme associated with the finite group $N_{\mcG}(\mcS)(k)/\mcT(k)$.
    (Indeed, it suffices to check this claim over the special fiber, which is well-known.)
    Using this fact, one can check that the map
    \[
    N_{\mcG}(\mcS)(\mfS[1/t])/\mcT(\mfS[1/t]) \to N_{\mcG}(\mcS)(F_i)/\mcT(F_i)
    \]
    is bijective.
    It then suffices to prove that
    \[
    \mcT(\mfS[1/t])/\mcT(\mfS) \to \mcT(F_i)/\mcT(\mcO_{i})
    \]
    is bijective.
    Since $\mcT_{W(k)^{\ur}}$ is split,
    we have a $\Gal(\overline{k}/k)$-equivariant bijection
    $
    \mcT(\mfS^{\ur}[1/t])/\mcT(\mfS^{\ur}) \simlr \mcT(F^{\ur}_i)/\mcT(\mcO_{F^{\ur}_i})
    $
    where $\mfS^{\ur} := \mfS \otimes_{W(k)} W(k)^{\ur}$.
    Thus it is enough to show that
    \[
    H^1(\Gal(\overline{k}/k), \mcT(\mfS^{\ur})) \to H^1(\Gal(\overline{k}/k), \mcT(\mcO_{F^{\ur}_i}))
    \]
    is injective.
    This is clear since a $\mcT$-torsor over $\Spec \mfS$ is trivial if its restriction to $\Spec \mcO_{i}$ is trivial.
\end{proof}

For our purpose, we need an analogous construction for Bruhat--Tits group schemes for $G_F$.
Here we closely follow the arguments of \cite[Section 4]{PappasZhu} and \cite[Section 2]{FHLR25}.
Let $i \in \mbN \cup \{ \infty \}$.
Let
\[
\mcA(G_i, S_i) \subset \mcB(G_i)
\]
be the apartment associated with $S_i$ in the (reduced) Bruhat--Tits building $\mcB(G_i)$ of $G_i$.
By \Cref{Lemma:Iwahori Weyl group}, the Iwahori--Weyl group $\widetilde{W}$ naturally acts on $\mcA(G_i, S_i)$.
The pinning
$\{ x_\alpha \}_{\alpha \in \Delta}$
gives rise to a hyperspecial vertex
$x_{0, i} \in \mcA(G_i, S_i)$.
Since $\mcA(G_i, S_i)$ is an affine space over $X_{*}(S_{\mathrm{sc}, i}) \otimes_{\mbZ} \mbR$, the bijections
\[
X_{*}(S_{\mathrm{sc}, \infty}) \overset{\sim}{\longleftarrow} \Hom_{\mfS^{\ur}}(\mbG_{m}, (\mcS_{\mathrm{sc}})_{\mfS^{\ur}}) \simlr X_{*}(S_{\mathrm{sc}, i})
\]
induce a natural identification
\begin{equation}\label{equation:identification of apartments}
    \mcA(G_\infty, S_\infty) \simeq \mcA(G_i, S_i)
\end{equation}
which sends $x_{0, \infty}$ to $x_{0, i}$.
This identification is $\widetilde{W}$-equivariant and also identifies the affine roots on both sides.
Similarly,
for the reductive group $\mcG_{K_0((t))}$, where $K_0:=W(k)[1/p]$,
we have a natural identification
\begin{equation}\label{equation:identification of apartments for laurant powerseries}
    \mcA(G_\infty, S_\infty) \simeq \mcA(\mcG_{K_0((t))}, \mcS_{K_0((t))}).
\end{equation}
Let $x_\infty \in \mcA(G_\infty, S_\infty)$ be any point.
Then, via \eqref{equation:identification of apartments} and \eqref{equation:identification of apartments for laurant powerseries},
we have corresponding points $x_i \in \mcA(G_i, S_i)$ for all $i \in \mbN$ and $x_{K_0((t))} \in \mcA(\mcG_{K_0((t))}, \mcS_{K_0((t))})$.
Let $\mathcal{P}_{x_\infty}, \mathcal{P}_{x_i}, \mathcal{P}_{x_{K_0((t))}}$ be the associated Bruhat--Tits group schemes (with connected fibers) over $\mcO_\infty$, $\mcO_i$, $K_0[[t]]$, respectively.
To the compatible family
$x=\{ x_i \}_{i \in \mbN \cup \{ \infty \}}$, one can attach a variant of Bruhat--Tits group schemes as follows.

\begin{proposition}\label{Proposition: Bruhat-Tits group scheme over W(k)[[t]]}
There exists a unique smooth affine group scheme model $\mcP'_{x}$
of $\mcG_{\mfS[1/t]}$ over $\mfS$
with connected fibers satisfying the following properties:
\begin{enumerate}
    \item For each $i \in \mbN \cup \{ \infty \}$, the base change of $\mcP'_{x}$ along the map $\mfS \to \mcO_{i}$ is $\mathcal{P}_{x_i}$.
    \item The base change of $\mcP'_{x}$ along the map $\mfS \to K_0[[t]]$ is $\mathcal{P}_{x_{K_0((t))}}$.
\end{enumerate}
\end{proposition}

\begin{proof}
    We first prove the uniqueness of $\mcP'_{x}$.
    Let $\mcP'_{x} = \Spec A$.
    Since $A$ is flat over $\mfS$, we may regard $A$, $A\otimes_{\mfS} K_0[[t]]$ and $A[1/t]$ as subrings of $A\otimes_{\mfS} K_0((t))$,
    and we have
    $
    A = (A\otimes_{\mfS} K_0[[t]]) \cap (A[1/t])
    $
    (since $\mfS = K_0[[t]] \cap \mfS[1/t]$).
    The uniqueness of $A$ then follows.

    In order to construct $\mcP'_{x}$, we may assume that $\mcT$ is split by Galois descent and uniqueness.
    Let $R({G_\infty, T_\infty})$ be the set of roots for $(G_\infty, T_\infty)$, and let
    $f \colon R({G_\infty, T_\infty}) \to \mbR$
    be the concave function associated with $x_\infty$.
    Then, by the same argument as in \cite[Proposition 2.6]{FHLR25}, we can show that the models $\mcT_{\mfS}$ and $t^{f(a)}\mathbb{G}_{a}$ for all $a \in R({G_\infty, T_\infty})$
    birationally glue to a smooth affine group scheme model
    $\mcP'_{x}$
    of $\mcG_{\mfS[1/t]}$ over $\mfS$
    with connected fibers.
    (We note that $f(a) \in \Z$ by construction.)
    Comparing with the construction of
    Bruhat--Tits group schemes, we see that the model $\mcP'_{x}$ satisfies the desired properties.
\end{proof}

\begin{Corollary}\label{Corollary:Bruhat-Tits group scheme over O}
    There exists a unique smooth affine group scheme model $\mcP_{x}$
of $G_F$ over $\mcO$
with connected fibers such that
for each $i \in \mbN \cup \{ \infty \}$, the base change of $\mathcal{P}_{x}$ along the map $\mcO \to \mcO_{i}$ is $\mathcal{P}_{x_i}$.
Moreover, there exists a closed embedding $\mathcal{P}_{x} \hookrightarrow \GL_n$ of group schemes over $\mcO$ such that the fppf quotient $\GL_n/\mathcal{P}_{x}$ is (representable by) a smooth quasi-affine $\mcO$-scheme.
\end{Corollary}

We call $\mcP_{x}$ the \textit{Bruhat--Tits group scheme over $\mcO$} associated with $x=\{ x_i \}_{i \in \mbN \cup \{ \infty \}}$.

\begin{proof}
    The uniqueness follows from \Cref{Lemma: uniqueness of integral model} below.
    By \Cref{Proposition: Bruhat-Tits group scheme over W(k)[[t]]}, the base change $\mathcal{P}_{x} := \mathcal{P}'_{x} \times_{\Spec \mfS} \Spec \mcO$, where $\mfS \to \mcO$ is defined by $t \mapsto \pi$, satisfies the desired properties.
    The last statement follows from \cite[Corollary 11.7]{PappasZhu} (by applying it to $\mcP'_{x}$).
\end{proof}

The following lemma is used in the proof of \Cref{Corollary:Bruhat-Tits group scheme over O}.

\begin{Lemma}\label{Lemma: uniqueness of integral model}
    Let $A$ be a finitely presented $F$-algebra.
    Let $R, R' \subset A$ be finitely presented $\mcO$-subalgebras such that $R[1/\pi] = R'[1/\pi] =A$.
    We assume that $R_{\mcO_{\infty}}=R'_{\mcO_{\infty}}$ in $A_{F_\infty}$.
    Then there exists an integer $B \geq 1$ such that
    $R_{\mcO_{\geq B}}=R'_{\mcO_{\geq B}}$ in $A_{F_{\geq B}}$.
\end{Lemma}

\begin{proof}
    It suffices to show that
    $R_{\mcO_{\geq B}} \subset R'_{\mcO_{\geq B}}$ in $A_{F_{\geq B}}$ for some $B \geq 1$.
    Since $R$ is a finitely generated $\mcO$-algebra, it suffices to show that for an element $x \in R$, there exists some $B \geq 1$ such that $x \in R'_{\mcO_{\geq B}}$.
    Let $n \geq 1$ be such that $y:=\pi^nx \in R'$.
    It suffices to prove $\pi^{n-1}x \in R'_{\mcO_{\geq B}}$ for some $B$, or equivalently, $\overline{y}=0$ in $R'_{\mcO_{\geq B}/\pi}$.
    Since $R'_{\mcO/\pi}$ is finitely presented over $\mcO/\pi$, there exists an $\mcO_{[1, B], 1}$-algebra $R'_B$ for some $B$ such that
    $
    R'_B \otimes_{\mcO_{[1, B], 1}} \mcO/\pi \simeq R'_{\mcO/\pi}
    $
    and $\overline{y} \in R'_B$.
    Since $\overline{y}=0$ in $R'_{\mcO_{\infty}/\pi_\infty}$ by assumption, it follows that $\overline{y}=0$ in 
    $R'_{\mcO_{\geq B+1}/\pi}$.
\end{proof}

\subsection{Parahoric Hecke algebras in families of close local fields}
\label{Subsection:Parahoric Hecke algebras in families of close local fields}

In this subsection, we consider parahoric Hecke algebras in families of close local fields.
We keep the notation of \Cref{Subsection:Unramified reductive groups in families}.
We assume that the residue field $k$ is finite.
Let
$\mathbf{a}_i \subset \mcA(G_i, S_i)$
be the standard chamber (corresponding to the pinning $\{ x_\alpha \}_{\alpha \in \Delta}$), so that the hyperspecial vertex $x_{0, i}$ lies in the closure $\overline{\mathbf{a}}_i$.
Let $I_i \subset G(F_i)$ be the corresponding Iwahori subgroup.
As in \Cref{Subsection:Unramified reductive groups in families}, we choose a compatible family
$x=\{ x_i \}_{i \in \mbN \cup \{ \infty \}}$
of points $x_i \in \mcA(G_i, S_i)$
and let $J_i:=\mcP_{x_i}(\mcO_i) \subset G(F_i)$ be the corresponding parahoric subgroup.
Without loss of generality, we may assume that $x_i \in \overline{\mathbf{a}}_i$.
Then $I_i \subset J_i$.

We recall that
the chamber $\mathbf{a}_i \subset \mcA(G_i, S_i)$
determines a splitting
\[
\widetilde{W}_i = W_{\mathrm{aff}, i} \rtimes (\widetilde{W}_i)_{\mathbf{a}_i}
\]
where $W_{\mathrm{aff}, i}$ is the subgroup generated by the set $\Delta_{\mathrm{aff}, i}$ of affine simple reflections along the walls bounding $\mathbf{a}_i$ and $(\widetilde{W}_i)_{\mathbf{a}_i}$ is the stabilizer of $\mathbf{a}_i$.
Let $(W_{\mathrm{aff}, i})_{x_i} \subset W_{\mathrm{aff}, i}$ be the stabilizer of $x_i$.
Then, by \cite[Theorem 7.8.1]{KPbook}, we have natural bijections
\[
\widetilde{W}_i \simlr I_i \backslash G(F_i) / I_i \quad \text{and} \quad 
(W_{\mathrm{aff}, i})_{x_i}\backslash\widetilde{W}_i/(W_{\mathrm{aff}, i})_{x_i} \simlr J_i \backslash G(F_i) / J_i.
\]

Let $\mcH_{J_i}(G(F_i))$ be the parahoric Hecke algebra associated with $J_i$.
This is a $\mathbb{C}$-algebra under convolution $\ast$, 
with respect to the Haar measure $dg$ on $G(F_i)$ normalized so that $J_i$ has volume $1$.
In the special case where $J_i=I_i$,
we write $\cdot$ for the convolution $\ast$ on the Iwahori--Hecke algebra $\mcH_{I_i}(G(F_i))$.
A $\mathbb{C}$-basis of $\mcH_{J_i}(G(F_i))$ is given by the indicator functions of the double cosets $J_iwJ_i$, where $w \in (W_{\mathrm{aff}, i})_{x_i}\backslash\widetilde{W}_i/(W_{\mathrm{aff}, i})_{x_i}$; we will denote these functions by $T_{w}=1_{J_iwJ_i}$.
By \Cref{Lemma:Iwahori Weyl group},
we have a natural isomorphism
$\widetilde{W}_\infty \simlr \widetilde{W}_i$
for each $i \in \mbN$, which induces a bijection
\[
(W_{\mathrm{aff}, \infty})_{x_\infty}\backslash\widetilde{W}_\infty/(W_{\mathrm{aff}, \infty})_{x_\infty}
\simlr
(W_{\mathrm{aff}, i})_{x_i}\backslash\widetilde{W}_i/(W_{\mathrm{aff}, i})_{x_i}, \quad  w_\infty \mapsto w_i.
\]
Thus, we have the following isomorphism of $\mbC$-vector spaces:
\begin{equation}\label{equation: parahoric Hecke algebras}
    \psi_i \colon \mcH_{J_\infty}(G(F_\infty)) \simlr \mcH_{J_i}(G(F_i)), \quad T_{w_\infty} \mapsto T_{w_i}.
\end{equation}

\begin{proposition}\label{Proposition:algebra isomorphism of parahoric Hecke algebras}
    The map $\psi_i \colon \mcH_{J_\infty}(G(F_\infty)) \simlr \mcH_{J_i}(G(F_i))$ is an isomorphism of $\mbC$-algebras for all $i \in \mbN$.
\end{proposition}

\begin{proof}
Although this is probably well-known to specialists, we provide a sketch of the argument.
    We first treat Iwahori--Hecke algebras.
    By the Iwahori--Matsumoto presentation given in \cite{RostamiBernsteinGeneralReductive},
    in order to see that $\psi_i$ is a homomorphism of $\mbC$-algebras in this case,
    it suffices to show that
    $[I_\infty s_\infty I_\infty \colon I_\infty]=[I_i s_i I_i \colon I_i]$
    for all $s_\infty \in \Delta_{\mathrm{aff}, \infty}$.
    For this, we may assume that $\mathcal{G}=\mathcal{G}_{\mathrm{sc}}$, and then
    $[I_i s_i I_i \colon I_i]$ can be expressed in terms of affine root systems as in \cite[Section 3.3.1]{TitsReductiveGroups}, which shows that it is independent of $i \in \mbN \cup \{ \infty \}$.

    Now we consider the general case.
    The parahoric Hecke algebra $\mcH_{J_i}(G(F_i))$ is naturally a $\mbC$-vector subspace of $\mcH_{I_i}(G(F_i))$.
    The isomorphisms $\psi_i$ are compatible with the embeddings $\mcH_{J_i}(G(F_i)) \hookrightarrow \mcH_{I_i}(G(F_i))$.
    For two functions $f, f' \in \mcH_{J_i}(G(F_i))$, we have $[J_i \colon I_i](f \ast f') = f \cdot f'$.
    It follows from the discussion in the previous paragraph that the map $\psi_i \colon \mcH_{J_\infty}(G(F_\infty)) \simlr \mcH_{J_i}(G(F_i))$
    preserves multiplication up to a scalar multiple.
    Since $\psi_i$ sends $T_1$ to $T_1$, this scalar multiple must in fact be equal to $1$.
    Hence $\psi_i$ is a $\mbC$-algebra isomorphism for all $i$.
\end{proof}

\begin{Remark}\label{Remark:work of Ganapathy}
    Analogous statements for congruence subgroups of Iwahori subgroups or special maximal parahoric subgroups are already known in several cases; see \cite{KazhdanCloseField, GanapathyGSp4, GanapathyHeckealgebra, LH}.
    In such situations, it may be necessary to exclude finitely many $i$.
\end{Remark}

\section{Affine Grassmannian in families of close fields}\label{Section:Affine Grassmannian in families of close fields}

In this section, we introduce the affine Grassmannian in our setting and prove its representability
and finiteness following the by now standard arguments in the literature.

\subsection{$\mcO$-Witt vector affine Grassmannian}\label{Subsection:O-Witt vector affine Grassmannian}

We start by recalling some background on perfect schemes.
We say that a morphism $f \colon X \to Y$ of perfect schemes over $\mbF_p$ is \textit{locally perfectly finitely presented}
if for any affine open subscheme $V=\Spec A \subset Y$ and any affine open subscheme $U=\Spec B \subset X$ with $f(U) \subset V$, the map $A \to B$ is perfectly finitely presented in the sense of \cite[Definition 3.10]{BhattScholzeProjectivity}, i.e.\ $B = (B_0)_{\perf}$ for some finitely presented $A$-algebra $B_0$.
If $f$ is moreover qcqs (quasi-compact and quasi-separated), then we say that $f$ is \textit{perfectly finitely presented} (or \textit{pfp} for short); see \cite[Proposition 3.11]{BhattScholzeProjectivity} for equivalent definitions.
For a functor $f \colon \mcF \to \mcF'$ on $\Perf_{\mcO/\pi}$ which is relatively representable by a morphism of perfect schemes, we say that $f$ is pfp if it is relatively representable by a pfp morphism.

\begin{Definition}\label{Definition:affine Grassmannian GLn}
    Let $n$ be a positive integer.
Let $\Gr_{\GL_n}$ be the functor on
$\Perf_{\mcO/\pi}$
sending $R$ to the set of isomorphism classes of pairs
$(M, \iota)$
where $M$ is a finite projective $W_\mcO(R)$-module of rank $n$ and $\iota$ is an isomorphism $\iota \colon M[1/\pi] \overset{\sim}{\to} W_\mcO(R)[1/\pi]^{\oplus n}$.
We also call such a pair $(M, \iota)$ a lattice $M$ in $W_\mcO(R)[1/\pi]^{\oplus n}$.
\end{Definition}

We shall prove that the functor $\Gr_{\GL_n}$ is an increasing union of perfections of finitely presented projective $\mcO/\pi$-schemes (see \Cref{Theorem:affine Grassmannian of GLn is representable}). To ensure the finite presentation property, we first note the following lemma:

\begin{Lemma}\label{Lemma:lattices inclusion closed condition}
    Let $R$ be a perfect $\mcO/\pi$-algebra.
    Let $M_1$ and $M_2$ be finite projective $W_\mcO(R)$-modules
    with an identification
    $M_1[1/\pi] \simeq M_2[1/\pi]$.
    There exists a finitely generated ideal $I \subset R$ such that a homomorphism
    $R \to R'$ of perfect $\mcO/\pi$-algebras factors through $(R/I)_{\perf}$ if and only if $(M_1)_{W_{\mcO}(R')} \subset (M_2)_{W_{\mcO}(R')}$.
\end{Lemma}

\begin{proof}
    We have $M_1 \subset \pi^{-m} M_2$ for some  positive integer $m$.
    By induction on $m$, we may assume that $M_1 \subset \pi^{-1} M_2$.
    Then $M_1 \subset M_2$ if and only if
    the image of $M_1$ in the finite projective $R$-module $M:=\pi^{-1} M_2/M_2$ is zero.
    We embed $M$ into $R^n$ as a direct summand for some $n$.
    Let $f_1, \dotsc, f_l$ be generators of $M_1$.
    Then let $I \subset R$ be the ideal generated by the entries of the images of $f_1, \dotsc, f_l$ in $R^n$.
    We easily see that this $I$ satisfies the desired property.
\end{proof}

As in \cite{ZhuAffineGrassmannianMixed} and \cite{BhattScholzeProjectivity}, we first show the representability of a subfunctor $\overline{\Gr}_{\GL_n, N} \subset \Gr_{\GL_n}$ that one might think of as the space of isogenies of height $N$.

\begin{proposition}\label{Proposition:Schbert varieties and Demazure resolutions for GLn}
    Let $N$ be a positive integer.
    Consider the subfunctor
    $\overline{\Gr}_{\GL_n, N} \subset \Gr_{\GL_n}$
    sending $R$ to the set of isomorphism classes of 
    lattices $M$ in $W_\mcO(R)[1/\pi]^{\oplus n}$ which are contained in $W_\mcO(R)^{\oplus n}$ and satisfy $\bigwedge^nM \simeq \pi^NW_\mcO(R)$.
Then $\overline{\Gr}_{\GL_n, N}$ is representable by the perfection of a finitely presented projective $\mcO/\pi$-scheme.
\end{proposition}

\begin{proof}
    Let $\widetilde{\Gr}_{\GL_n, N}$ be the functor on
    $\Perf_{\mcO/\pi}$ sending $R$ to the set of isomorphism classes of $N$-tuples
    \[
    M_N \subset \cdots \subset M_1 \subset M_0=W_\mcO(R)^{\oplus n}
    \]
    of lattices (of rank $n$) such that every $M_j/M_{j+1}$ is an invertible $R$-module.
    By the same arguments as in \cite[Proposition 7.11 and Proposition 8.6]{BhattScholzeProjectivity}, we see that
    the functor
    $\widetilde{\Gr}_{\GL_n, N} \to \overline{\Gr}_{\GL_n, N}$
    defined by $\{ M_j \} \mapsto M_N$ is relatively representable by a proper surjective pfp morphism, and that $\widetilde{\Gr}_{\GL_n, N}$
    is representable by the perfection of a smooth projective $\mcO/\pi$-scheme.
    Let $\mcL$ be the line bundle on $\widetilde{\Gr}_{\GL_n, N}$
    given by
    $
    \bigotimes^{N-1}_{j=0} M_j/M_{j+1}.
    $
    Let $\widetilde{\Gr}_{\GL_n, N, i}$ and $\mcL_i$ denote the base changes of $\widetilde{\Gr}_{\GL_n, N}$
    and $\mcL$ along the map $\mcO/\pi \to \mcO_i/\pi_i=k$ for $i \in \mbN \cup \{ \infty \}$.
    Since $\widetilde{\Gr}_{\GL_n, N}$ is pfp over $\mcO/\pi = \varinjlim_B \mcO_{[1, B], 1}$ and each $\mcO_{[1, B], 1} \to \mcO/\pi$ has a splitting,
    there is some $B \geq 1$ such that
    $(\widetilde{\Gr}_{\GL_n, N}, \mcL)$ is the base change of
    \[
    (\coprod_{i \leq B} (\widetilde{\Gr}_{\GL_n, N, i}, \mcL_i)) \coprod (\widetilde{\Gr}_{\GL_n, N, \infty}, \mcL_\infty)
    \]
    along $\mcO_{[1, B], 1} \to \mcO/\pi$ (see \cite[Proposition 3.12]{BhattScholzeProjectivity}).
    By the same arguments as in the proof of \cite[Theorem 8.3]{BhattScholzeProjectivity}, 
    one can show that $\mcL_i$ is semiample on $\widetilde{\Gr}_{\GL_n, N, i}$ for each $i$, and we then take the Stein factorization 
    $\phi \colon \widetilde{\Gr}_{\GL_n, N} \to X$,
    where $X$ is the perfection of a finitely presented projective $\mcO/\pi$-scheme and $\phi$ is a proper surjective pfp morphism.
    We consider two closed subschemes
    \[
    \widetilde{\Gr}_{\GL_n, N} \times_X \widetilde{\Gr}_{\GL_n, N} \quad \text{and} \quad
    \widetilde{\Gr}_{\GL_n, N} \times_{\overline{\Gr}_{\GL_n, N}} \widetilde{\Gr}_{\GL_n, N}
    \]
    of $\widetilde{\Gr}_{\GL_n, N} \times_{\Spec \mcO/\pi} \widetilde{\Gr}_{\GL_n, N}$.
    By the same argument as in the proof of \cite[Theorem 8.3]{BhattScholzeProjectivity},
    these two subschemes coincide with each other after restricting to perfect $\mcO_i/\pi_i$-algebras.
    Since they are pfp over $\mcO/\pi$, it follows that 
    \[\widetilde{\Gr}_{\GL_n, N} \times_X \widetilde{\Gr}_{\GL_n, N}=
    \widetilde{\Gr}_{\GL_n, N} \times_{\overline{\Gr}_{\GL_n, N}} \widetilde{\Gr}_{\GL_n, N}.
    \]
    As explained in \textit{loc.cit.},
    this implies that $X \simeq \overline{\Gr}_{\GL_n, N}$.
    In particular $\overline{\Gr}_{\GL_n, N}$ is representable by the perfection of a finitely presented projective $\mcO/\pi$-scheme.
\end{proof}

Now we can show the representability of $\Gr_{\GL_n}$.

\begin{Theorem}\label{Theorem:affine Grassmannian of GLn is representable}
    The functor $\Gr_{\GL_n}$ is an increasing union of (subfunctors which are representable by) perfections of finitely presented projective $\mcO/\pi$-schemes along closed immersions.
\end{Theorem}

\begin{proof}
    The functor $\Gr_{\GL_n}$ is an increasing union of subfunctors
    $\Gr^{[-m, m]}_{\GL_n}$
    parametrizing lattices $M$ in $W_\mcO(R)[1/\pi]^{\oplus n}$ between $\pi^mW_\mcO(R)^{\oplus n}$ and $\pi^{-m}W_\mcO(R)^{\oplus n}$.
    The inclusion $\Gr^{[-m, m]}_{\GL_n} \subset \Gr_{\GL_n}$
    is (relatively representable by) a 
pfp closed immersion by \Cref{Lemma:lattices inclusion closed condition}.
Moreover $\Gr^{[-m, m]}_{\GL_n}$ is a disjoint union of finitely many open and closed subfunctors which admit a pfp closed immersion into $g\overline{\Gr}_{\GL_n, N} \subset \Gr_{\GL_n}$ for some $g \in \GL_n(\mcO[1/\pi])$ and some large $N$.
Thus the assertion follows from \Cref{Proposition:Schbert varieties and Demazure resolutions for GLn}.
\end{proof}

Let $\mcP$ be a smooth affine group scheme over $\Spec \mcO$
with a closed embedding $\iota_\mcP \colon \mcP \hookrightarrow \GL_n$ of group schemes over $\mcO$ such that the fppf quotient $\GL_n/\mcP$ is (representable by) a smooth quasi-affine $\mcO$-scheme.

\begin{proposition}\label{Proposition:etale descent for torsors}
    The fibered category over $\Perf^{\op}_{\mcO/\pi}$ which associates to a perfect $\mcO/\pi$-algebra $R$ the category of $\mcP$-torsors over $\Spec W_\mcO(R)$ satisfies \'etale descent.
\end{proposition}

\begin{proof}
    This follows from arguments analogous to those in the proofs of \cite[Proposition B.0.2]{Bultel-Pappas} and \cite[Proposition 1.2.6]{Zhuintroaffinegrassmannians}.
\end{proof}

\begin{Definition}\label{Definition:affine Grassmannian of parahoric}
    Let $\Gr_{\mcP}$ be the functor on $\Perf_{\mcO/\pi}$ sending $R$ to the set of isomorphism classes of pairs
$(\mcE, \iota)$
where $\mcE$ is a $\mcP$-torsor over $\Spec W_\mcO(R)$
and $\iota$ is a trivialization
$
\iota \colon \mcE_{W_\mcO(R)[1/\pi]} \overset{\sim}{\to} \mcP_{W_\mcO(R)[1/\pi]}.
$
The functor $\Gr_{\mcP}$ forms an \'etale sheaf  (by \Cref{Proposition:etale descent for torsors}), which is canonically isomorphic to the quotient sheaf
$L_{\mcO}\mcP/L^+_{\mcO}\mcP$.
\end{Definition}

When $\mcP=\mcP_x$ is the Bruhat--Tits group scheme over $\mcO$ constructed in \Cref{Corollary:Bruhat-Tits group scheme over O}, the functor $\Gr_{\mcP}$ is also called the \textit{partial affine flag variety}.

As usual, we may deduce the representability of $\Gr_{\mcP}$ from the known result for $\GL_{n}$.

\begin{Corollary}\label{Corollary:affine Grassmannian of parahoric is representable}
The induced functor
    $\Gr_{\mcP} \to \Gr_{\GL_n}$
    is a pfp locally closed immersion.
    If $\mcP=\mcP_x$ is the Bruhat--Tits group scheme over $\mcO$ constructed in \Cref{Corollary:Bruhat-Tits group scheme over O}, then the functor $\Gr_{\mcP}$ is an increasing union of perfections of finitely presented projective $\mcO/\pi$-schemes along closed immersions.
\end{Corollary}

\begin{proof}
    The first assertion follows from the same argument as in \cite[Proposition 1.2.6]{Zhuintroaffinegrassmannians}.
    We then see that $\Gr_{\mcP}$ is an increasing union of perfections of 
    finitely presented 
    \textit{quasi-projective} $\mcO/\pi$-schemes along closed immersions.
    If $\mcP=\mcP_x$, then these schemes are in fact \textit{projective} since $\Gr_{\mcP}$ is known to be ind-proper after restricting to each $\mcO_i/\pi_i$ by \cite[Section 1.4.2]{ZhuAffineGrassmannianMixed}.
\end{proof}

We need to discuss how to decide whether certain locally closed sub-ind-schemes of $\Gr_{\mcP}$ are actually schemes.

\begin{Definition}\label{Definition:bounded subset}
We write $\mcP=\Spec A$.
        A subset $\Omega \subset \mcP(F)$ is \textit{bounded} if for all $f\in A,$ there is a positive integer $N$ such that the set $\lbrace f(b)\in F \colon  b\in \Omega \rbrace$ is contained in $\pi^{-N}\mcO.$
\end{Definition}

\begin{Remark}\label{Remark:equivalent conditions for bounded subsets}
    Let $i \colon \mcP_F \hookrightarrow \mathbb{A}^{m}_{F}$ be a closed immersion of $F$-schemes.
    For a subset $\Omega \subset \mcP(F)$, the following conditions are equivalent.
    \begin{enumerate}
        \item $\Omega \subset \mcP(F)$ is bounded.
        \item 
        There exists an $N > 0$ such that $i(\Omega) \subset F^{m}$ is contained in $\pi^{-N}\mcO^m.$
        \item There exists an $N > 0$ such that for all $g \in \iota_\mcP(\Omega) \subset \GL_{n}(F)$, we have
        $\pi^N \mcO^n \subset g(\mcO^n) \subset \pi^{-N} \mcO^n$.
    \end{enumerate}
\end{Remark}

Let $\Gr_{\mcP} \hookrightarrow \Gr_{\GL_n}$ be the pfp locally closed immersion as in \Cref{Corollary:affine Grassmannian of parahoric is representable}.
    As in the proof of \Cref{Theorem:affine Grassmannian of GLn is representable}, let $\Gr^{[-N, N]}_{\GL_n} \subset \Gr_{\GL_n}$
    be the subfunctor parametrizing lattices $M$ in $W_\mcO(R)[1/\pi]^{\oplus n}$ between $\pi^NW_\mcO(R)^{\oplus n}$ and $\pi^{-N}W_\mcO(R)^{\oplus n}$.

\begin{proposition}\label{Proposition: quasi-compact sub-ind-scheme of affine Grassmannian}
    Let $Z \hookrightarrow \Gr_{\mcP}$ be a pfp locally closed immersion (so that $Z$ is a sub-ind-scheme).
    The following assertions are equivalent.
    \begin{enumerate}
        \item $Z$ is quasi-compact.
        \item The preimage of $Z(\breve{\mcO}/\pi)$ in $\mcP(\breve{F})$ 
    under the projection
    $\mcP(\breve{F}) \to \Gr_{\mcP}(\breve{\mcO}/\pi)=\mcP(\breve{F})/\mcP(\breve{\mcO})$
    is bounded in the sense of \Cref{Definition:bounded subset}.
        \item $Z$ is contained in $\Gr^{[-N, N]}_{\GL_n}$ for some $N$.
    \end{enumerate}
    Moreover, if $Z$ satisfies these equivalent conditions, then $Z$ is a pfp $\mcO/\pi$-scheme.
\end{proposition}

\begin{proof}
    By the proof of \Cref{Theorem:affine Grassmannian of GLn is representable}, $\Gr^{[-N, N]}_{\GL_n}$ is the perfection of a finitely presented projective $\mcO/\pi$-scheme.
    Then it is clear that $(1)$ is equivalent to $(3)$ and $Z$ is a scheme if $Z$ satisfies these equivalent conditions.
    In order to prove that $(2) \Leftrightarrow (3)$, we may assume that $k= \overline{k}$.
    The claim then follows from \Cref{Remark:equivalent conditions for bounded subsets}.
\end{proof}

\subsection{Schubert varieties}\label{Subsection:Schubert varieties}

Let $\mcP:=\mcP_{x}$ be the smooth affine group scheme over $\mcO$ constructed in \Cref{Corollary:Bruhat-Tits group scheme over O}.
We will discuss Schubert varieties for the affine Grassmannian $\Gr_{\mcP}$ and deduce that they are locally constant in families of close fields (see \Cref{Corollary: Schubert variety local constant}).

Let $\widetilde{W}$ be the Iwahori--Weyl group of $G$; see \Cref{Subsection:Unramified reductive groups in families}.
The map $\mfS \to \mcO$, $t \mapsto \pi$ induces a map
$\widetilde{W} \to \Gr_{\mcP}(\mcO/\pi)$.
Let $w \in \widetilde{W}$ be an element with image $\overline{w} \in \Gr_{\mcP}(\mcO/\pi)$.
Then the $L^+_\mcO \mcP$-orbit
\[
S^\circ_w:= L^+_\mcO \mcP \cdot \overline{w}  \subset \Gr_{\mcP}
\]
is a pfp locally closed subscheme of $\Gr_{\mcP}$
(i.e.\ $S^\circ_w \hookrightarrow \Gr_{\mcP}$ is a locally closed immersion and $S^\circ_w$ is pfp over $\mcO/\pi$).
More precisely, we have the following result.

\begin{proposition}\label{Proposition:Schubert varieties}
    Let $S_w \subset \Gr_{\mcP}$ be the Zariski closure of $S^\circ_w$.
    Then $S_w$ is the perfection of a finitely presented projective $\mcO/\pi$-scheme and $S^\circ_w \subset S_w$ is a quasi-compact open subset.
\end{proposition}

\begin{proof}
    There is an $n \geq 1$ such that the action of $L^+_\mcO \mcP$ on $\overline{w}$ factors through $L^n_\mcO \mcP$.
    By \Cref{Lemma:representability of L^n X}, $L^n_\mcO \mcP$ is finitely presented over $\mcO/\pi$.
    Let $Z \subset \Gr_{\mcP}$ be a closed subscheme which is pfp over $\mcO/\pi$ such that
    $S^\circ_w \subset Z$.
    Consider the map
    $L^n_\mcO \mcP \to Z$, $g \mapsto g\overline{w}$.
    Then for some $B \geq 1$, this map can be identified with the base change of the natural map
    \[
    (\coprod_{i \leq B} L^n_{\mcO_i} \mcP_i) \coprod L^n_{\mcO_\infty} \mcP_\infty \to (\coprod_{i \leq B} Z_i) \coprod Z_\infty
    \]
    along $\mcO_{[1, B], 1} \to \mcO/\pi$.
    Thus the assertion follows from the corresponding statements for $\Gr_{\mcP_i}$ ($i \in \mbN \cup \{ \infty \}$).
\end{proof}

For $i \in \mbN \cup \{ \infty \}$, we denote by $S^\circ_{w_i} \subset S_{w_i} \subset \Gr_{\mcP_i}$ the base change of $S^\circ_{w} \subset S_{w} \subset \Gr_{\mcP}$ along $\mcO/\pi \to \mcO_i/\pi_i=k$.
These are the (perfect) Schubert cell and the (perfect) Schubert variety in
$\Gr_{\mcP_i}$, respectively, associated with the image $w_i \in \widetilde{W}_i$ of $w$.
We note that for $i=\infty$, these objects are defined as finite type $k$-schemes, but in our context we only consider their perfections.

\begin{Corollary}[Local constancy of Schubert varieties]\label{Corollary: Schubert variety local constant}
    For each $w \in \widetilde{W}$,
    there is some $B \geq 1$ (depending on $w$) such that for all $i \geq B$, we have an isomorphism
    $S_{w_i} \simeq S_{w_\infty}$
    which induces 
    $S^\circ_{w_i} \simeq S^\circ_{w_\infty}$.
\end{Corollary}

\begin{proof}
    Since $S_w$ is finitely presented over $\mcO/\pi$, there exists some $B \geq 1$ such that
    $S_w$
    is the base change of
    $(\coprod_{i \leq B-1} S_{w_i}) \coprod S_{w_\infty}$
    along the map $\mcO_{[1, B-1], 1} \to \mcO/\pi$.
    It then follows that $S_{w_i} \simeq S_{w_\infty}$ for all $i \geq B$.
    After increasing $B$ if necessary, we may assume that this isomorphism induces $S^\circ_{w_i} \simeq S^\circ_{w_\infty}$.
\end{proof}

\begin{Remark}\label{Remark:Bando}
    Bando \cite{BandoGeoSatakeSpringer} also proved local constancy results for Schubert varieties as in \Cref{Corollary: Schubert variety local constant} in the hyperspecial case.
    His method also implies that the bound $B$ can in principle be computed, whereas in our method, we are not able to control $B$.
\end{Remark}

\section{Affine Springer fibers in families of close fields}\label{Section: Affine Springer fibers in families of close fields}

In this section, we introduce affine Springer fibers in our setting and show that certain quotients of these affine Springer fibers are perfect algebraic spaces that are perfectly finitely presented over $\mcO/\pi$.
As mentioned in the introduction, we also introduce a twisted version of affine Springer fibers and prove analogous results.
These results will be crucially used later to verify the local constancy of twisted $\kappa$-orbital integrals in families of close local fields.

\subsection{Perfectly finitely presented algebraic spaces}\label{Subsection: Perfectly finitely presented algebraic spaces}

Let $X$ be an algebraic space over $\mbF_p$.
Let $\Frob_X \colon X \to X$ be the Frobenius morphism, i.e.\ for a scheme $S$ over $\mbF_p$, the map $\Frob_X \colon X(S) \to X(S)$ is induced by the Frobenius endomorphism $\Frob_S \colon S \to S$.
We say that $X$ is \textit{perfect} if $\Frob_X$ is an isomorphism.

\begin{Remark}\label{Remark:perfect algebraic space}
    Let $f \colon U \to X$ be an \'etale morphism where $U$ is a scheme.
    Then the relative Frobenius morphism $(\Frob_U, f) \colon U \to U \times_{X, \Frob_X} X$ is an isomorphism.
    Thus the following assertions are equivalent:
    \begin{enumerate}
        \item $X$ is perfect.
        \item Every scheme $U$ \'etale over $X$ is perfect.
        \item There exists an \'etale surjection $U \to X$ where $U$ is a perfect scheme.
    \end{enumerate}
\end{Remark}

For a morphism $f \colon X \to Y$ of perfect schemes, being locally perfectly finitely presented is \'etale local on the source and target.
Thus we can extend the definition of locally perfectly finitely presented morphisms to morphisms of perfect algebraic spaces over $\mbF_p$.

\begin{proposition}\label{Proposition: pfp morphisms and limit preserving functors}
    Let $f \colon X \to Y$ be a morphism of perfect algebraic spaces over $\mbF_p$.
    Then $f$ is locally perfectly finitely presented if and only if for any cofiltered system $(Z_i, f_{ij})$ of qcqs perfect algebraic spaces over $Y$ with affine transition maps,
    the natural map
        $
        \varinjlim_i \Mor_{Y}(Z_i, X) \to \Mor_Y(\varprojlim_i Z_i, X)
        $
        is bijective.
        (The limit $\varprojlim_i Z_i$ is a (perfect) algebraic space over $Y$ by \cite[Tag 07SF]{stacks}.)
\end{proposition}

\begin{proof}
    By a standard argument as in \cite[Tag 04AK and Tag 0CP4]{stacks}, one can deduce this from the corresponding statement for perfect schemes \cite[Proposition 3.11]{BhattScholzeProjectivity}.
\end{proof}

A morphism $f \colon X \to Y$ of perfect algebraic spaces over $\mbF_p$ is called perfectly finitely presented (or pfp for short) if it is qcqs and locally perfectly finitely presented.
Let
$\mathrm{AS}^{\mathrm{pfp}}_{/ X}$ denote the category of perfect algebraic spaces $Z$ over $\mbF_p$ with a pfp morphism $Z \to X$.

\begin{proposition}\label{Proposition: category of pfp algebraic spaces}
    Let $(X_i, f_{ij})$ be a cofiltered system of qcqs perfect algebraic spaces over $\mbF_p$ with affine transition maps.
    Let $X:=\varprojlim_i X_i$.
    Then the natural functor
    $
    {2-\varinjlim}_{i} \mathrm{AS}^{\mathrm{pfp}}_{/ X_i} \to \mathrm{AS}^{\mathrm{pfp}}_{/ X}
    $
    is an equivalence.
\end{proposition}

\begin{proof}
By a standard argument as in \cite[Tag 07SK]{stacks}, one can deduce this from the corresponding statement for perfect schemes \cite[Proposition 3.12]{BhattScholzeProjectivity}.
\end{proof}

\subsection{$\mcO$-Witt vector affine Springer fibers}\label{Subsection:O-Witt vector affine Springer fibers}

We work in the setting of \Cref{Subsection:Unramified reductive groups in families}.
Let $\mcP:=\mcP_x$ be the smooth affine group scheme over $\mcO=\mcO_F$ as in \Cref{Corollary:Bruhat-Tits group scheme over O}.
Let $Z \subset \Gr_{\mcP}$ be a pfp locally closed subscheme which is stable under the action of $L^+_\mcO \mcP$ by left multiplication.
We will use $Z$ to bound certain denominators.
Let $\gamma \in G(F)$ be a fiberwise strongly regular semisimple element, so that the centralizer
$G_\gamma$ is a maximal torus of $G$ (\Cref{Lemma: centralizer of fiberwise strongly regular semisimple}).

\begin{Definition}[Affine Springer fiber]\label{def: Affine Springer fiber O}
    The \textit{affine Springer fiber} associated with
    the triple
    $(\mcP, \gamma, Z)$ is
    the subfunctor
    \[
    X^Z_{\mcP, \gamma} \subset \Gr_{\mcP}=L_\mcO \mcP/L^+_\mcO \mcP
    \]
    such that, for any perfect $\mcO/\pi$-algebra $R$, 
    $X^Z_{\mcP, \gamma}(R)$ consists of elements
    $g \in \Gr_{\mcP}(R)$ satisfying $\overline{g^{-1}_x \gamma g_x} \in Z(\kappa(x))$ for every geometric point $x$ of $\Spec R$, where $g_x \in L_\mcO \mcP(\kappa(x))=\mcP(W_\mcO(\kappa(x))[1/\pi])$ is a lift of the image of $g$ in $\Gr_{\mcP}(\kappa(x))=L_\mcO \mcP(\kappa(x))/L^+_\mcO \mcP(\kappa(x))$.
    Note that $X^Z_{\mcP, \gamma} \hookrightarrow \Gr_{\mcP}$ is a pfp locally closed immersion.
\end{Definition}

For the $F$-torus $G_\gamma$, we assume that we are in \Cref{Situation: torus} and use the notation there.
We denote by
$
\Lambda_\gamma \subset \Hom_{E_\infty}(\mbG_{m}, (G_\gamma)_{E_\infty})
$
the subgroup consisting of elements fixed by the inertia subgroup of $\Gal(E_\infty/F_\infty)$.
With the notation of \Cref{Construction: etale covering of family of close fields},
we have 
\[
\Lambda_\gamma=\Hom_{F'_\infty}(\mbG_{m}, (G_\gamma)_{F'_\infty}).
\]
We embed $\Lambda_\gamma$ into 
$\Hom_{E_{\geq B}}(\mbG_{m}, (G_\gamma)_{E_{\geq B}})$
via the splitting \eqref{equation:lattice in cocharacter group}
\[
\Hom_{E_\infty}(\mbG_{m}, (G_\gamma)_{E_\infty}) \to \Hom_{E_{\geq B}}(\mbG_{m}, (G_\gamma)_{E_{\geq B}}).
\]
Since $\Lambda_\gamma$
is contained in $\Hom_{F'_{\geq B}}(\mbG_{m}, (G_\gamma)_{F'_{\geq B}})$,
we then obtain an injection
\[
\Lambda_\gamma \to L_\mcO G_\gamma(\mcO'_{\geq B}/\pi), \quad \lambda \mapsto \lambda(\pi).
\]
We will identify $\Lambda_\gamma$ with the image of this injection.
Since $L_\mcO G_\gamma$ acts on $X^Z_{\mcP, \gamma}$ by left multiplication, this injection induces an action of $\Lambda_\gamma$ on
$(X^Z_{\mcP, \gamma})_{\mcO'_{\geq B}/\pi}$.
We note that this action is compatible with the actions of $\Gal(E_\infty/F_\infty)$.

\begin{proposition}\label{Proposition: quasi-compact subscheme Y in ASF}
    There exists a pfp closed subscheme $Y \subset (X^Z_{\mcP, \gamma})_{\mcO'_{\geq B}/\pi}$ such that
    \[
    (X^Z_{\mcP, \gamma})_{\mcO'_{\geq B}/\pi}=\bigcup_{\lambda \in \Lambda_\gamma} \lambda(\pi) \cdot Y.
    \]
\end{proposition}

\begin{proof}
    We note that the corresponding statement for a single $F_i$ $(i \in \mbN)$ follows either from an argument similar to that in \cite[Section 2]{KazhdanLusztig}, or by directly applying the proof given below for families of close fields.
    Therefore, in the proof below, 
    we may freely enlarge $B$ (often without explicit mention).
    Moreover, by (unramified) Galois descent, we may assume that
    $\mcG$ is split and also that $E$ is totally ramified over $F$, that is, $\mcO'=\mcO$.

    \textit{Step 1.}
    We first assume that $x_\infty=x_{0, \infty}$ (with the notation of \Cref{Subsection:Unramified reductive groups in families}) so that
    $\mcP=\mcP_x$
    is the split reductive group scheme $\mcG_\mcO$.
    For simplicity, we write $\mcG=\mcG_\mcO$.
    We furthermore assume that 
    $G_\gamma$ is equal to the split maximal torus $T$ in \Cref{Situation: unramified reductive group in families of close fields}.
    Let $U$ be the unipotent radical of the Borel subgroup $B \subset G$.\footnote{By abuse of notation, we use the same letter $B$ to denote both the Borel subgroup and the bound $B \ge 1$ for the indices of close fields. This should not cause any confusion in context.}
    Let $Y_U$ be the intersection of the $L_\mcO U$-orbit of $\overline{1} \in \Gr_\mcG$ with $X^Z_{\mcG, \gamma}$.
    Then $Y_U \hookrightarrow X^Z_{\mcG, \gamma}$ is a pfp locally closed immersion.
    By the Iwasawa decomposition (on each fiber $F_i$),
    we have 
    $(X^Z_{\mcG, \gamma})_{\mcO_{\geq B}/\pi}=\bigcup_{\lambda \in \Lambda_\gamma} \lambda(\pi) \cdot Y_U$.

    We claim that $Y_U$ is quasi-compact.
    For this, by \Cref{Proposition: quasi-compact sub-ind-scheme of affine Grassmannian},
    it is enough to prove that the set
    $
    \{ u \in U(\breve{F}) \colon u^{-1} \gamma u \in \Omega  \}
    $
    is bounded in $G(\breve{F})$, where
    $\Omega \subset G(\breve{F})$ is the preimage of $Z(\breve{\mcO}/\pi)$ under the projection
    $G(\breve{F}) \to \Gr_{\mcG}(\breve{\mcO}/\pi)$.
    Let us choose an ordering on the basis $\Delta$ and use this to define
    the reverse lexicographical ordering on
    the set of positive roots
    $\lbrace \alpha_{1}, \dotsc, \alpha_{m} \rbrace$.
    Let $U_{i} \subset U$ be the root group corresponding to $\alpha_i$
    with a trivialization $q_i \colon \mbG_a \simeq U_i$.
    The multiplication map
    $\prod^m_{i=1}U_i \to U$ is an isomorphism, and moreover 
    $U_{\leq j}:=\prod^j_{i=1}U_i$ is a closed subgroup scheme of $U$ and is normalized by $U_{j+1}$; see \cite[Proposition 5.1.14]{ConradReductiveGroupSchemes}.
    Given $u=u_{1}\cdots u_{m} \in U(\breve{F})$ with $u_{i} = q_i(c_i) \in U_i(\breve{F})$, we define
    $v_i:=u^{-1}_i (\gamma u_i \gamma^{-1}) \in U_i(\breve{F})$.
    Moreover we define
    $v'_1:=v_1$ and $v'_{i} := (u^{-1}_i v'_{i-1} u_i) v_i$; then $v'_{m}= (u^{-1} \gamma u) \gamma^{-1}$.
    When $u^{-1} \gamma u$ is contained in the bounded subset $\Omega$, then the term $v'_{m}$ is bounded.
    Since 
    $u^{-1}_m v'_{m-1} u_m \in U_{\leq m-1}$
    and $v_m \in U_m$, it follows that both of these terms are bounded.
    Using that $v_m=q_m((\alpha_m(\gamma)-1)c_m)$ and
    that $\alpha_m(\gamma)-1$ is nonzero on each fiber $F_i$ (since $\gamma$ is fiberwise strongly regular semisimple), we see that the term $u_m=q_m(c_m)$ is also bounded.
    This implies that $v'_{m-1}$ is also bounded.
    By repeating this process, we get the desired bound on all $u_{i}$.
    This proves our claim.
    Then the desired statement holds for the Zariski closure $Y$ of $Y_U$.

    \textit{Step 2.}
    We next drop the assumption on $G_\gamma$.
    We can assume without loss of generality that 
    $(G_\gamma)_{E_{\geq B}}$ and the split maximal torus $T_{E_{\geq B}}$ are $G(E_{\geq B})$-conjugate.
    Since $\Gr_{\Res^{\mcO_E}_\mcO \mcG_{\mcO_E}} = \Gr_{\mcG_{\mcO_E}}$, 
    the affine Grassmannian $\Gr_{\mcG_{\mcO_E}}$
    admits a natural action of $\Gal(E_\infty/F_\infty)$
    and there is a natural pfp closed immersion
    $\Gr_\mcG \hookrightarrow \Gr_{\mcG_{\mcO_E}}$
    with $\Gr_\mcG \subset (\Gr_{\mcG_{\mcO_E}})^{\Gal(E_\infty/F_\infty)}$.
    We choose a pfp closed subscheme $\widetilde{Z} \subset \Gr_{\mcG_{\mcO_E}}$ which is stable under the actions of $L^+_{\mcO_E}\mcG_{\mcO_E}$ and $\Gal(E_\infty/F_\infty)$ such that $Z \subset \widetilde{Z}$.
    Let $\widetilde{\Lambda}_\gamma:=\Hom_{E_\infty}(\mbG_{m}, (G_\gamma)_{E_\infty})$.
    It follows from the result proved in the previous case that there exists a pfp closed subscheme $\widetilde{Y} \subset (X^{\widetilde{Z}}_{\mcG_{\mcO_E}, \gamma})_{\mcO_{\geq B}/\pi}$ such that
    \begin{equation}\label{equation:widetilde lambda orbits}
        (X^{\widetilde{Z}}_{\mcG_{\mcO_E}, \gamma})_{\mcO_{\geq B}/\pi}=\bigcup_{\lambda \in \widetilde{\Lambda}_\gamma} \lambda(\pi_E) \cdot \widetilde{Y}.
    \end{equation}
    Here $\pi_E \in \mcO_E$ is the element defined in \Cref{Construction: etale covering of family of close fields}.
    We have $\pi=u\pi^r_E$ for some $u \in \mcO^\times_E$.
    By enlarging $\widetilde{Y}$, we may assume that $\lambda(u)\cdot \widetilde{Y}=\widetilde{Y}$ for all $\lambda \in \widetilde{\Lambda}_\gamma$.
    Let $\lambda_1, \dotsc, \lambda_m \in \widetilde{\Lambda}_\gamma$
    be a complete set of representatives for $\widetilde{\Lambda}_\gamma/r \widetilde{\Lambda}_\gamma$.
    Then, by replacing $\widetilde{Y}$ with $\bigcup^{m}_{i=1} \lambda_i(\pi_E) \cdot \widetilde{Y}$, we may assume that
    \eqref{equation:widetilde lambda orbits}
    remains valid after replacing $\pi_E$ by $\pi$.
    After enlarging $\widetilde{Y}$, we may assume further that $\widetilde{Y}$ is $\Gal(E_\infty/F_\infty)$-stable.
    Let $S := \{  \lambda \in \widetilde{\Lambda}_\gamma \colon \lambda(\pi) \cdot \widetilde{Y} \cap \widetilde{Y} \neq \emptyset \}$, which is a finite set.
    There exists a finite subset $D  \subset \widetilde{\Lambda}_\gamma$
    such that
    \[
    \Lambda_\gamma+D = \{ \lambda \in \widetilde{\Lambda}_\gamma \colon g(\lambda)-\lambda \in S,  \, \forall g \in \Gal(E_\infty/F_\infty) \}.
    \]
    Now we define
    $
    Y := (X^{Z}_{\mcG, \gamma})_{\mcO_{\geq B}/\pi} \cap (\bigcup_{\lambda\in D} \lambda(\pi)\cdot\widetilde{Y}),
    $
    for which one can check that $(X^Z_{\mcG, \gamma})_{\mcO_{\geq B}/\pi}=\bigcup_{\lambda \in \Lambda_\gamma} \lambda(\pi) \cdot Y.$
    
    \textit{Step 3.}
    Finally, we drop the assumption on $\mcP=\mcP_x$.
    If $\mcP_x$ is the standard Iwahori subgroup $\mcI$ (i.e.\ $x$ lies in the standard chamber), then the desired statement can be deduced from the case of $\mcG$ since the natural map
    $\Gr_{\mcI} \to \Gr_{\mcG}$ is proper.
    For a general $\mcP_x$, we may assume that $x$ is contained in the closure of the standard chamber, and then we are reduced to the case of $\mcI$ since
    the natural map
    $\Gr_{\mcI} \to \Gr_{\mcP}$ is proper.
\end{proof}

\begin{Corollary}\label{Corollary: ASF is a scheme}
    The affine Springer fiber $X^Z_{\mcP, \gamma}$ is a perfect scheme which is locally perfectly finitely presented over $\mcO/\pi$.
\end{Corollary}

\begin{proof}
    The corresponding statement for a single $F_i$ $(i \in \mbN)$ follows from \cite[Theorem 5.4]{ChiWittAffineSpringer} (or by directly applying the proof given for families of close fields here).
    Thus it suffices to prove the statement for $(X^Z_{\mcP, \gamma})_{\mcO_{\geq B}/\pi}$ for some $B$.
    We can deduce this from \Cref{Proposition: quasi-compact subscheme Y in ASF} and \Cref{Proposition: quasi-compact sub-ind-scheme of affine Grassmannian} by the same argument as in \cite[Section 2.5.3]{YunLectureNote}.
\end{proof}

Let $\underline{\Lambda}_\gamma$ be the \'etale group scheme over $k$ attached to the $\Gal(k'/k)$-module $\Lambda_\gamma$,
and let us denote
its base change to $\Spec \mcO_{\geq B}/\pi$ by the same notation.
The action of $\Lambda_\gamma$ on
$(X^Z_{\mcP, \gamma})_{\mcO'_{\geq B}/\pi}$ 
naturally extends to an action of $\underline{\Lambda}_\gamma$ on $(X^Z_{\mcP, \gamma})_{\mcO_{\geq B}/\pi}$.

The main result of this section is the following theorem.

\begin{Theorem}\label{Theorem: ASF is quasi-compact separated pfp algebraic space}
The quotient $\underline{\Lambda}_\gamma \backslash (X^Z_{\mcP, \gamma})_{\mcO_{\geq B}/\pi}$ is a separated pfp perfect algebraic space over $\Spec \mcO_{\geq B}/\pi$.
\end{Theorem}

\begin{proof}
    The map
    $\underline{\Lambda}_\gamma \times (X^Z_{\mcP, \gamma})_{\mcO_{\geq B}/\pi} \to (X^Z_{\mcP, \gamma})_{\mcO_{\geq B}/\pi} \times (X^Z_{\mcP, \gamma})_{\mcO_{\geq B}/\pi}$, $(g, x) \mapsto (x, g\cdot x)$
    is a closed immersion.
    Thus the quotient $\underline{\Lambda}_\gamma \backslash (X^Z_{\mcP, \gamma})_{\mcO_{\geq B}/\pi}$ is a separated perfect algebraic space
    which is locally perfectly finitely presented 
    over $\mcO_{\geq B}/\pi$.
    The fact that $\underline{\Lambda}_\gamma \backslash (X^Z_{\mcP, \gamma})_{\mcO_{\geq B}/\pi}$ is quasi-compact follows from \Cref{Proposition: quasi-compact subscheme Y in ASF}.
\end{proof}

\subsection{$\mcO$-Witt vector twisted affine Springer fibers}\label{Subsection:O-Witt vector twisted affine Springer fibers}

We introduce the geometric objects relevant for the geometric interpretation of twisted $\kappa$-orbital integrals.

\begin{Situation}\label{Situation: twisted affine Springer fiber}
We assume that $k$ is a finite field.
Let $\widetilde{k} \subset \overline{k}$ be a finite extension of $k$ of degree $r$.
For a $W(k)$-algebra $R$, we write
$\widetilde{R} = R \otimes_{W(k)} W(\widetilde{k})$.
For a smooth affine group scheme $H$ over $\Spec R$, we write
$H_{\widetilde{R}} = H \times_{\Spec R} \Spec \widetilde{R}$
and
$\widetilde{H}:= \Res^{\widetilde{R}}_R (H_{\widetilde{R}})$ by a slight abuse of notation.
We have that
$
(\widetilde{H})_{\widetilde{R}} \simeq H_{\widetilde{R}}\times \cdots \times H_{\widetilde{R}},
$
where the factors are indexed by $\Gal(\widetilde{k}/k)$ and the $j$-th entry corresponds to $\sigma^{j}\in \Gal(\widetilde{k}/k)$ for $0 \leq j \leq r-1.$
Here $\sigma \in \Gal(\widetilde{k}/k)$ is the Frobenius element.
Let $s \colon \widetilde{H} \to \widetilde{H}$ be the automorphism induced by $\sigma$ (i.e.\ for all $R$-algebras $A$, $s\colon \widetilde{H}(A)=H(A\otimes_{R} \widetilde{R})\rightarrow \widetilde{H}(A)=H(A\otimes_{R} \widetilde{R})$ is given by $\id \otimes \sigma$), which corresponds to
 $
(x_{0},\dotsc,x_{r-1})\mapsto (x_{1},\dotsc,x_{r-1}, x_{0})
$
over $\widetilde{R}$.
Note that we have an embedding of $R$-group schemes $H\hookrightarrow \widetilde{H}$, and $H\simeq \widetilde{H}^{s=1}.$
\end{Situation}

Applying this notation to \Cref{Situation: unramified reductive group in families of close fields}, we obtain
the quasi-split reductive group scheme
$\widetilde{\mcG}=\Res^{W(\widetilde{k})}_{W(k)} (\mcG_{W(\widetilde{k})})$
over $W(k)$ with maximal torus $\widetilde{\mcT}$ and Borel subgroup $\widetilde{\mcB}$
and a natural $\Gal(\overline{k}/k)$-stable pinning of
$(\widetilde{\mcG}, \widetilde{\mcT}, \widetilde{\mcB})$
induced by $\{ x_\alpha \}_{\alpha \in \Delta}$.
We also have the corresponding objects
$(\widetilde{G}, \widetilde{T}, \widetilde{B})$ over $F$ and 
$(\widetilde{G}_i, \widetilde{T}_i, \widetilde{B}_i)$ over $F_i$.
Let $x = \{ x_i \}_{i \in \mbN \cup \{ \infty \}}$ be a compatible family of points $x_i \in \mcA(G_i, S_i)$ as in \Cref{Subsection:Unramified reductive groups in families}.
This naturally gives rise to a compatible family of points, which will also be denoted by $x = \{ x_i \}_{i \in \mbN \cup \{ \infty \}}$, in the corresponding apartments of $\mcB(\widetilde{G}_i)$, and the associated Bruhat--Tits group scheme over $\mcO$ in the sense of \Cref{Corollary:Bruhat-Tits group scheme over O} is isomorphic to $\widetilde{\mcP}_x= \Res^{\widetilde{\mcO}}_\mcO ((\mcP_x)_{\widetilde{\mcO}})$; see for example \cite[Section 7.9]{KPbook}.
We will simply write $\widetilde{\mcP}=\widetilde{\mcP}_x$.

\begin{Definition}\label{Definition: sigma centralizer}
    Let $\delta \in \widetilde{G}(F)=G(\widetilde{F})$ be an element.
    We assume that $N(\delta):=\delta \cdot s(\delta)\cdots s^{r-1}(\delta)\in G(\widetilde{F})$
    is a fiberwise strongly regular semisimple element.
The \textit{$\sigma$-centralizer} of $\delta$ is 
the closed $F$-subgroup $\widetilde{G}_{\delta s} \subset \widetilde{G}$
defined by
$\widetilde{G}_{\delta s}(A) = \{ g \in \widetilde{G}(A) \colon g\delta s(g)^{-1}= \delta \}$.
The $0$-th projection
$(\widetilde{G})_{\widetilde{F}} \simeq G_{\widetilde{F}}\times \cdots \times G_{\widetilde{F}} \to G_{\widetilde{F}}$
induces an isomorphism
$(\widetilde{G}_{\delta s})_{\widetilde{F}} \simlr (G_{\widetilde{F}})_{N(\delta)}$ (cf.\ \cite[Lemma 5.4]{KottwitzRationalConjugacy}).
Thus $\widetilde{G}_{\delta s}$ is a torus over $F$.
\end{Definition}

\begin{Definition}[Twisted affine Springer fiber]\label{Definition: twisted affine Springer fiber}
    Let $Z \subset \Gr_{\widetilde{\mcP}}$ be a pfp locally closed subscheme which is stable under the action of $L^+_\mcO \widetilde{\mcP}$.
The \textit{twisted affine Springer fiber} associated with
    the triple
    $(\widetilde{\mcP}, \delta, Z)$ is
    the subfunctor
    \[
    X^Z_{\widetilde{\mcP}, \delta s} \subset \Gr_{\widetilde{\mcP}}=L_\mcO \widetilde{\mcP}/L^+_\mcO \widetilde{\mcP}
    \]
    such that, for any perfect $\mcO/\pi$-algebra $R$, 
    $X^Z_{\widetilde{\mcP}, \delta s}(R)$ consists of elements
    $g \in \Gr_{\widetilde{\mcP}}(R)$
    satisfying $\overline{g^{-1}_x \delta s(g_x)} \in Z(\kappa(x))$ for every geometric point $x$ of $\Spec R$, where $g_x \in L_\mcO \widetilde{\mcP}(\kappa(x))$ is a lift of the image of $g$ in $\Gr_{\widetilde{\mcP}}(\kappa(x))$.
    We note that 
    $X^Z_{\widetilde{\mcP}, \delta s}  \hookrightarrow \Gr_{\widetilde{\mcP}}$ is a pfp locally closed immersion.
\end{Definition}

\begin{Remark}\label{Remark: twisted ASF is a generalization of ASF}
    If $r=1$, then $s= \id$, so twisted affine Springer fibers coincide with affine Springer fibers introduced in \Cref{def: Affine Springer fiber O}.
    In this case, we will use the notation $\gamma$ instead of $\delta$; then $\widetilde{G}_{\delta s}=G_\gamma$.
\end{Remark}

For the $F$-torus $\widetilde{G}_{\delta s}$, we assume that we are in \Cref{Situation: torus} and use the notation there.
Let
$
\Lambda_{\delta s} \subset \Hom_{E_\infty}(\mbG_{m}, (\widetilde{G}_{\delta s})_{E_\infty})
$
be the subgroup consisting of elements fixed by the inertia subgroup of $\Gal(E_\infty/F_\infty)$.
The image of $\Lambda_{\delta s}$ under the splitting \eqref{equation:lattice in cocharacter group}
\[
\Hom_{E_\infty}(\mbG_{m}, (\widetilde{G}_{\delta s})_{E_\infty}) \to \Hom_{E_{\geq B}}(\mbG_{m}, (\widetilde{G}_{\delta s})_{E_{\geq B}})
\]
is contained in $\Hom_{F'_{\geq B}}(\mbG_{m}, (\widetilde{G}_{\delta s})_{F'_{\geq B}})$.
Thus we have an injection
$
\Lambda_{\delta s} \to L_\mcO \widetilde{G}_{\delta s}(\mcO'_{\geq B}/\pi), \lambda \mapsto \lambda(\pi).
$
We will identify $\Lambda_{\delta s}$ with the image of this injection.
Since $L_\mcO \widetilde{G}_{\delta s}$ acts on $X^Z_{\widetilde{\mcP}, \delta s}$ by left multiplication, this injection induces an action of $\Lambda_{\delta s}$ on
$(X^Z_{\widetilde{\mcP}, \delta s})_{\mcO'_{\geq B}/\pi}$.

The results of the previous subsection can be extended to the twisted setting as follows.

\begin{Theorem}\label{Theorem: finiteness of twisted affine Springer fiber}
     \begin{enumerate}
         \item There exists a pfp closed subscheme $Y \subset (X^Z_{\widetilde{\mcP}, \delta s})_{\mcO'_{\geq B}/\pi}$ such that
    \[
    (X^Z_{\widetilde{\mcP}, \delta s})_{\mcO'_{\geq B}/\pi}=\bigcup_{\lambda \in \Lambda_{\delta s}} \lambda(\pi) \cdot Y.
    \]
    \item $X^Z_{\widetilde{\mcP}, \delta s}$ is a perfect scheme which is locally perfectly finitely presented over $\mcO/\pi$.
    \item Let $\underline{\Lambda}_{\delta s}$ be the \'etale group scheme over $\mcO_{\geq B}/\pi$ attached to $\Lambda_{\delta s}$.
    Then the quotient $\underline{\Lambda}_{\delta s} \backslash (X^Z_{\widetilde{\mcP}, \delta s})_{\mcO_{\geq B}/\pi}$ 
    is a separated pfp perfect algebraic space over $\mcO_{\geq B}/\pi$.
     \end{enumerate}
\end{Theorem}

\begin{proof}
    For (1), we may assume that $\widetilde{k} \subset k'$ by (unramified) Galois descent.
    Then $(\Gr_{\widetilde{\mcP}})_{\mcO'/\pi} = (\Gr_{\mcP})_{\mcO'/\pi} \times \cdots \times (\Gr_{\mcP})_{\mcO'/\pi}$.
    We choose a pfp locally closed subscheme
    $W \subset(\Gr_{\mcP})_{\mcO'/\pi}$
    which is stable under the action of $L^+_\mcO \mcP$ such that $Z_{\mcO'/\pi} \subset W \times  \cdots \times W$.
    Then $(X^Z_{\widetilde{\mcP}, \delta s})_{\mcO'/\pi}$ is contained in
    \[
    \{ (g_0, \dotsc, g_{r-1}) \in (\Gr_{\mcP})_{\mcO'/\pi} \times \cdots \times (\Gr_{\mcP})_{\mcO'/\pi} \colon \overline{g^{-1}_j \sigma^{j}(\delta) g_{j+1}} \in W \, (0 \leq j \leq r-1) \}
    \]
    where we write
    $\delta = (\delta_0, \dotsc, \delta_{r-1}) \in \widetilde{G}(F') = G(F') \times \cdots \times G(F')$
    and $g_r:=g_0$.
    Using this description, we can find a pfp closed subscheme $Z' \subset (\Gr_{\mcP})_{\mcO'/\pi}$ which is stable under the action of $L^+_\mcO \mcP$ such that
    the $0$-th projection
    $(\Gr_{\widetilde{\mcP}})_{\mcO'/\pi}  \to (\Gr_{\mcP})_{\mcO'/\pi}$
    induces a quasi-compact morphism
    $(X^Z_{\widetilde{\mcP}, \delta s})_{\mcO'/\pi} \to X^{Z'}_{\mcP_{\mcO', N(\delta)}}$, where the target is the affine Springer fiber associated with the triple
    $(\mcP_{\mcO'}, N(\delta), Z')$ (in the sense of \Cref{def: Affine Springer fiber O}).
    Now the assertion (1) can be deduced from \Cref{Proposition: quasi-compact subscheme Y in ASF}.

    The assertions (2) and (3) follow from (1) by the same arguments as in \Cref{Corollary: ASF is a scheme} and \Cref{Theorem: ASF is quasi-compact separated pfp algebraic space}.
\end{proof}

\section{Local constancy of twisted $\kappa$-orbital integrals}\label{Section: Local constancy of stable twisted orbital integrals}

We continue to work under the setup of \Cref{Subsection:O-Witt vector twisted affine Springer fibers}.
The purpose of this section is to prove that the twisted $\kappa_i$-orbital integrals
$TO^{\kappa_i}_{\delta_i}(f_i)$
are locally constant around $\infty$ for suitable choices of $\kappa_i$ and functions $f_i$ in the parahoric Hecke algebras associated with $\mcP_{x_i}(\widetilde{\mcO}_i)$.
In the sequel, we will only use the fact that both the $\kappa$-orbital integrals and the stable twisted orbital integrals are locally constant around $\infty.$
In order to treat these two situations simultaneously, we state and prove the local constancy for twisted $\kappa$-orbital integrals.

Without loss of generality, we may assume that $x = \{ x_i \}_{i \in \mbN \cup \{ \infty \}}$ lies in the closure of the standard chamber, and we write $\widetilde{J}_i:=\mcP_{x_i}(\widetilde{\mcO}_i)$
for the corresponding parahoric subgroup of $G(\widetilde{F}_i)$.

\subsection{Geometric interpretations of twisted $\kappa$-orbital integrals}\label{Subsection: Geometric interpretations of twisted kappa-orbital integrals}

We first recall the definition of twisted $\kappa$-orbital integrals.
In this subsection,
we fix an $i \in \mbN \cup \{ \infty \}$
and work over $F_i$ as the base field.
We will omit the subscript $i$ from the notation for simplicity.
Thus, we just write
$F=F_i$, $\widetilde{F}=\widetilde{F}_i$ and $\delta=\delta_i \in G(\widetilde{F})=\widetilde{G}(F)$.
Recall that $\widetilde{F}$ is an unramified extension of $F$ of degree $r$.
We also recall that $N(\delta) = \delta \cdot s(\delta) \cdots s^{r-1}(\delta) \in \widetilde{G}(F)$ is assumed to be strongly regular semisimple.

Following \cite{KottwitzRationalConjugacy}, for an element $\delta' \in \widetilde{G}(F)$, we say that $\delta$ and $\delta'$ are \textit{$F^{\sep}$-$\sigma$-conjugate}, if there exists an element $h\in \widetilde{G}(F^{\sep})$ such that
$
\delta'=h^{-1}\delta s(h).
$
We define
$$
\mcD(\widetilde{G}_{\delta s}/F):=\ker(H^{1}(F, \widetilde{G}_{\delta s})\to H^{1}(F, \widetilde{G})).
$$
Then $\mcD(\widetilde{G}_{\delta s}/F)$
is in bijection with
the set of $\sigma$-conjugacy classes inside the $F^{\sep}$-$\sigma$-conjugacy class of $\delta \in \widetilde{G}(F)$.
Let
$g \mapsto t_g \in \widetilde{G}_{\delta s}(F^{\sep})$
be a $1$-cocycle representing an element $t$ of 
$\mcD(\widetilde{G}_{\delta s}/F)$; thus there exists $h\in \widetilde{G}(F^{\sep})$ such that $t_g=h^{-1}g(h)
$
for any $g \in \Gamma=\Gal(F^{\sep}/F)$.
Then
$\delta(t):=h \delta s(h)^{-1}$ 
is invariant under the action of $\Gamma$ so that it lies in $\widetilde{G}(F)$.
By construction, $\delta(t)$ is $F^{\sep}$-$\sigma$-conjugate to $\delta$.
This sets up the desired bijection.

\begin{Remark}\label{Remark: D(G_gamma) and B(G_gamma)}
    The map
    $H^{1}(F, \widetilde{G}_{\delta s})\to H^{1}(F, \widetilde{G})$
of pointed sets
can be identified with
the homomorphism
$X_*(\widetilde{G}_{\delta s})_{\Gamma, \tors}=\pi_1(\widetilde{G}_{\delta s})_{\Gamma, \tors} \to \pi_1(\widetilde{G})_{\Gamma, \tors}$ (see for example \cite[Theorem 11.7.7]{KPbook}), and thus $\mcD(\widetilde{G}_{\delta s}/F)\subset H^1(F, \widetilde{G}_{\delta s})$ is a normal subgroup.
The natural homomorphism
$H^1(F, \widetilde{G}_{\delta s}) \to B(\widetilde{G}_{\delta s})$
induces a bijection
\[
\mcD(\widetilde{G}_{\delta s}/F) \simlr \ker(B(\widetilde{G}_{\delta s}) \to B(\widetilde{G})).
\]
(See also the discussion in \cite[Section 15.5]{GoreskyKottwitzMacPhersonUnramified}. Here for a reductive group $H$ over $F$, we let $B(H)$ denote the set of $\sigma$-conjugacy classes in $H(\breve{F})$.)
The Kottwitz homomorphism induces an isomorphism
$B(\widetilde{G}_{\delta s}) \simlr X_*(\widetilde{G}_{\delta s})_{\Gamma}$,
and the natural homomorphism
$H^1(F, \widetilde{G}_{\delta s}) \to B(\widetilde{G}_{\delta s})$ can be identified with the inclusion
$X_*(\widetilde{G}_{\delta s})_{\Gamma, \tors} \hookrightarrow X_*(\widetilde{G}_{\delta s})_{\Gamma}$; see \cite{KottwitzIso1} for details.
\end{Remark}

We normalize our Haar measures as follows: Let $dg$ be the Haar measure on $\widetilde{G}(F)=G(\widetilde{F})$ such that the volume of $\widetilde{J}$ is $1$
and
let $dg_{\delta}$ be the Haar measure 
on $\widetilde{G}_{\delta s}(F)$
such that
the volume of $\widetilde{G}_{\delta s}(F)^{0}$ is $1$, where $\widetilde{G}_{\delta s}(F)^{0} \subset \widetilde{G}_{\delta s}(F)$ is the kernel of the Kottwitz homomorphism.
For $f\in C^{\infty}_{c}(\widetilde{G}(F)),$
the twisted orbital integral
$TO_{\delta}(f)$
is defined by
\[
TO_{\delta}(f)  = \int_{\widetilde{G}_{\delta s}(F)\backslash \widetilde{G}(F)}f(g^{-1}\delta s(g))\frac{dg}{dg_{\delta}}.
\]
For an element $t \in \mcD(\widetilde{G}_{\delta s}/F)$,
we set $TO_{t}(f):=TO_{\delta(t)}(f)$
where $\delta(t) \in \widetilde{G}(F)$ is a representative of the $\sigma$-conjugacy class corresponding to $t$ in the $F^{\sep}$-$\sigma$-conjugacy class of $\delta$.
Let 
$\kappa \colon H^1(F, \widetilde{G}_{\delta s}) \simeq X_*(\widetilde{G}_{\delta s})_{\Gamma, \tors} \to \mbC^\times$
be a character.
The \textit{twisted $\kappa$-orbital integral}
$TO^{\kappa}_{\delta}(f)$ is then defined by
\[
TO^{\kappa}_{\delta}(f):= \sum_{t \in \mcD(\widetilde{G}_{\delta s}/F)} \langle t, \kappa \rangle^{-1}  \cdot TO_t(f),
\]
where $\langle t, \kappa \rangle \in \mbC^\times$ denotes the image of $t$ under $\kappa$.
If $\kappa=1$, then $STO_\delta(f):=TO^{1}_{\delta}(f)$ is called the \textit{stable twisted orbital integral}.
When $\widetilde{F}=F$, we will write
$\gamma:=\delta$
and
$O^{\kappa}_{\gamma}(f):=TO^{\kappa}_{\delta}(f)$, which we call the \textit{$\kappa$-orbital integral}.
We call $SO_\gamma(f):=O^{1}_{\gamma}(f)$ the \textit{stable orbital integral}.

Let $\widetilde{W}(\widetilde{G})$ be the Iwahori--Weyl group of $\widetilde{G}$, which is isomorphic to
the Iwahori--Weyl group of $G_{\widetilde{F}}$.
Let $w \in \widetilde{W}(\widetilde{G})$
and let $T_w = 1_{\widetilde{J}w\widetilde{J}} \in \mcH_{\widetilde{J}}(\widetilde{G}(F))$.
We shall give a geometric interpretation of $TO^\kappa_{\delta}(T_w)$, following a similar approach to that of \cite{GoreskyKottwitzMacPhersonUnramified}.
Let $X^w_{\widetilde{\mcP}, \delta s} \subset \Gr_{\widetilde{\mcP}}$ be the twisted affine Springer fiber associated with the triple
$(\widetilde{\mcP}, \delta, S^\circ_w)$ (defined in the same way as in \Cref{Definition: twisted affine Springer fiber}).
As in \Cref{Subsection:O-Witt vector twisted affine Springer fibers},
let $\Lambda_{\delta s}:=X_*(\widetilde{G}_{\delta s})^I$, which we identify with the image of the map
$\Lambda_{\delta s} \to \widetilde{G}_{\delta s}(F^{\ur})$, $\lambda \mapsto \lambda(\pi)$.
Let $\underline{\Lambda}_{\delta s}$ be the corresponding \'etale group scheme over $k$.
Then 
the quotient
$\underline{\Lambda}_{\delta s} \backslash X^w_{\widetilde{\mcP}, \delta s}$
is a separated pfp perfect algebraic space over $k$.
Let
$\mcD(\Lambda_{\delta s}/F) \subset \mcD(\widetilde{G}_{\delta s}/F)$
be the image of the homomorphism
$
\ker(H^1(k, \Lambda_{\delta s}) \to H^1(F, \widetilde{G})) \to \mcD(\widetilde{G}_{\delta s}/F).
$
Let
\[
c_1 := \vert \ker (H^1(k, \Lambda_{\delta s})\rightarrow H^1(F, \widetilde{G}_{\delta s})) \vert \quad \text{and} \quad 
c_2:=\frac{\vert \coker((\Lambda_{\delta s})_{\Gamma}\rightarrow X_{*}(\widetilde{G}_{\delta s})_{\Gamma}) \vert}{\vert \ker ((\Lambda_{\delta s})_{\Gamma}\rightarrow X_{*}(\widetilde{G}_{\delta s})_{\Gamma})\vert }.
\]

\begin{proposition}\label{Proposition:geometric interpretation of stable twisted orbital integrals}
We fix a character 
$B(\widetilde{G}_{\delta s}) \simeq X_*(\widetilde{G}_{\delta s})_{\Gamma} \to \mbC^\times$ whose restriction to $H^1(F, \widetilde{G}_{\delta s})$ is $\kappa$, and we denote it also by $\kappa$.
Let $\{ t_1, \dotsc, t_m \} \subset \mcD(\widetilde{G}_{\delta s}/F)$ be a complete set of representatives for the quotient $\mcD(\widetilde{G}_{\delta s}/F)/\mcD(\Lambda_{\delta s}/F)$.
Then for each $1 \leq j \leq m$, we have
\[
\sum_{\epsilon \in \mcD(\Lambda_{\delta s}/F)} \langle t_j\epsilon^{-1}, \kappa \rangle^{-1} \cdot TO_{t_j\epsilon^{-1}}(T_w) = \frac{\langle t_j, \kappa \rangle^{-1}}{c_1c_2} \sum_{x \in (\underline{\Lambda}_{\delta s} \backslash X^w_{\widetilde{\mcP}, \delta s})(\overline{k})^{t_j\sigma=1}} \langle \lambda_x, \kappa \rangle.
\]
The right-hand side is defined as follows:
We choose a lift of $t_j$ in $\widetilde{G}_{\delta s}(\breve{F})$, which we denote by the same notation $t_j$, and
$(\underline{\Lambda}_{\delta s} \backslash X^w_{\widetilde{\mcP}, \delta s})(\overline{k})^{t_j\sigma=1}$ is the set of fixed points of $t_j\sigma$ on $(\underline{\Lambda}_{\delta s} \backslash X^w_{\widetilde{\mcP}, \delta s})(\overline{k})$.
For any $x \in (\underline{\Lambda}_{\delta s} \backslash X^w_{\widetilde{\mcP}, \delta s})(\overline{k})^{t_j\sigma=1}$,
we choose a lift $\widetilde{x} \in X^w_{\widetilde{\mcP}, \delta s}(\overline{k})$
and let
$\lambda_x \in (\Lambda_{\delta s})_\Gamma$
be the image of the element $\lambda_{\widetilde{x}} \in \Lambda_{\delta s}$ such that $t_j\cdot \sigma(\widetilde{x})=\lambda_{\widetilde{x}} \cdot \widetilde{x}$.
(Notice that $\lambda_x$ is independent of the choice of $\widetilde{x}$.)

In particular, we have
\[
TO^\kappa_\delta(T_w)= \frac{1}{c_1c_2} \sum^m_{j=1} \left( \langle t_j, \kappa \rangle^{-1} \sum_{x \in (\underline{\Lambda}_{\delta s} \backslash X^w_{\widetilde{\mcP}, \delta s})(\overline{k})^{t_j\sigma=1}} \langle \lambda_x, \kappa \rangle
\right).
\]
\end{proposition}

\begin{proof}
   We fix $j$ and omit the subscript $j$ for simplicity.
    There exists an element $h\in \widetilde{G}(\breve{F})$
    such that 
    $t=h^{-1}\sigma(h).$
    Let
    $\delta(t):=h\delta s(h)^{-1} \in \widetilde{G}(F)$.
    Then we have an isomorphism
    $\widetilde{G}_{\delta s} \overset{\sim}{\to} \widetilde{G}_{\delta(t) s}$, $g \mapsto hgh^{-1}$, which is defined over $F$.
    Moreover, the map $L_{\mcO}\widetilde{G} \to L_{\mcO}\widetilde{G}$, $g \mapsto hg$, induces isomorphisms
    $(X^w_{\widetilde{\mcP}, \delta s})_{\overline{k}} \overset{\sim}{\to} (X^w_{\widetilde{\mcP}, \delta(t) s})_{\overline{k}}$ and
    $
    (\underline{\Lambda}_{\delta s} \backslash X^w_{\widetilde{\mcP}, \delta s})_{\overline{k}} \overset{\sim}{\to} (\underline{\Lambda}_{\delta(t) s} \backslash X^w_{\widetilde{\mcP}, \delta(t) s})_{\overline{k}}
    $
    over $\overline{k}$.
    The latter isomorphism induces a bijection
    \[
    (\underline{\Lambda}_{\delta s} \backslash X^w_{\widetilde{\mcP}, \delta s})(\overline{k})^{t\sigma=1}\overset{\sim}{\to} (\underline{\Lambda}_{\delta(t) s} \backslash X^w_{\widetilde{\mcP}, \delta(t) s})(k).
    \]
    The set $(\underline{\Lambda}_{\delta(t) s} \backslash X^w_{\widetilde{\mcP}, \delta(t) s})(k)$
    can be identified with the set of isomorphism classes of pairs
    $(\epsilon, f \colon P_\epsilon \to X^w_{\widetilde{\mcP}, \delta(t) s})$
    where
    \begin{itemize}
        \item $\epsilon \in H^1(k, \Lambda_{\delta s})$, which we view as an element of $H^1(k, \Lambda_{\delta(t) s})$ via the isomorphism $\Lambda_{\delta s} \simeq \Lambda_{\delta(t) s}$ induced by $g \mapsto hgh^{-1}$,
        \item $P_\epsilon$ is the corresponding $\underline{\Lambda}_{\delta(t) s}$-torsor over $\Spec k$, and 
        \item $f \colon P_\epsilon \to X^w_{\widetilde{\mcP}, \delta(t) s}$ is a $\underline{\Lambda}_{\delta(t)s}$-equivariant map over $k$.
    \end{itemize}
    We choose a $1$-cocycle $\Gal(\overline{k}/k) \to \Lambda_{\delta s}$ representing $\epsilon$, and denote by the same notation $\epsilon \in \Lambda_{\delta s}$ the image of $\sigma$.
    Then we may assume that $P_\epsilon(\overline{k})=\Lambda_{\delta(t)s}$ and the action of $\sigma$ is determined by $x \mapsto h\epsilon h^{-1}\sigma(x)$.
    If there is a $\underline{\Lambda}_{\delta(t)s}$-equivariant map $f \colon P_\epsilon \to X^w_{\widetilde{\mcP}, \delta(t) s}$, then
    the element $y:=f(1) \in X^w_{\widetilde{\mcP}, \delta(t) s}(\overline{k})$
    satisfies $\sigma(y)=h \epsilon h^{-1} y$.
    This implies that for the point $x \in (\underline{\Lambda}_{\delta s} \backslash X^w_{\widetilde{\mcP}, \delta s})(\overline{k})^{t\sigma=1}$ corresponding to a pair $(\epsilon, f) \in (\underline{\Lambda}_{\delta(t) s} \backslash X^w_{\widetilde{\mcP}, \delta(t) s})(k)$, we have $\lambda_x=\epsilon$.
    Moreover, from the equality $\sigma(y)=h \epsilon h^{-1} y$, one can deduce that 
    $h \epsilon^{-1} h^{-1} = h'^{-1}\sigma(h')$
    for some $h' \in \widetilde{G}(\breve{F})$; this follows from the same argument as in the last paragraph of the proof of \cite[Theorem 15.8]{GoreskyKottwitzMacPhersonUnramified}.
    Then
    $t\epsilon^{-1}=(h'h)^{-1}\sigma(h'h)$.
    In particular $t\epsilon^{-1}$ is trivial in $B(\widetilde{G})$, or equivalently $\epsilon \in \ker(H^1(k, \Lambda_{\delta s}) \to H^1(F, \widetilde{G}))$.
    We set $\delta(t\epsilon^{-1}):=h'\delta(t)s(h')^{-1} \in \widetilde{G}(F)$.
    The set of isomorphism classes of
    $\underline{\Lambda}_{\delta(t)s}$-equivariant maps $f \colon P_\epsilon \to X^w_{\widetilde{\mcP}, \delta(t) s}$
    is naturally identified with
    the set
    $\Lambda^{\Gamma}_{\delta(t\epsilon^{-1})s}\backslash (X^w_{\widetilde{\mcP}, \delta(t\epsilon^{-1})s}(k))$.
    On the other hand, by the same argument as in \cite[Section 15.4]{GoreskyKottwitzMacPhersonUnramified}, we see that
    \[
    TO_{t\epsilon^{-1}}(T_w)=TO_{\delta(t\epsilon^{-1})}(T_w)=\frac{1}{c_2} \vert \Lambda^{\Gamma}_{\delta(t\epsilon^{-1})s}\backslash (X^w_{\widetilde{\mcP}, \delta(t\epsilon^{-1})s}(k)) \vert.
    \]
    Putting everything together, we obtain
    \begin{align*}
    \frac{\langle t, \kappa \rangle^{-1}}{c_1c_2} \sum_{x \in (\underline{\Lambda}_{\delta s} \backslash X^w_{\widetilde{\mcP}, \delta s})(\overline{k})^{t_j\sigma=1}} \langle \lambda_x, \kappa \rangle
    &= \frac{\langle t, \kappa \rangle^{-1}}{c_1c_2} \sum_{\epsilon \in \ker(H^1(k, \Lambda_{\delta s}) \to H^1(F, \widetilde{G}))} 
    \langle \epsilon, \kappa \rangle \vert \Lambda^{\Gamma}_{\delta(t\epsilon^{-1})s}\backslash (X^w_{\widetilde{\mcP}, \delta(t\epsilon^{-1})s}(k)) \vert \\
    &= \frac{1}{c_1} \sum_{\epsilon \in \ker(H^1(k, \Lambda_{\delta s}) \to H^1(F, \widetilde{G}))} \langle t\epsilon^{-1}, \kappa \rangle^{-1} \cdot TO_{t\epsilon^{-1}}(T_w) \\
    &= \sum_{\epsilon \in \mcD(\Lambda_{\delta s}/F)} \langle t\epsilon^{-1}, \kappa \rangle^{-1} \cdot TO_{t\epsilon^{-1}}(T_{w}).
    \end{align*}
\end{proof}

\begin{Remark}\label{Remark: relation to affine DL varieties}
It is known (see \cite{ZhouZhuOrbitalIntegralsADLV} for example) that twisted orbital integrals can be described using affine Deligne--Lusztig varieties. However, our twisted affine Springer fibers are distinct from affine Deligne--Lusztig varieties in general, and the former may be more useful for describing \textit{stable} twisted orbital integrals.
\end{Remark}

\subsection{Local constancy of twisted $\kappa$-orbital integrals}\label{Subsection:Local constancy of stable twisted orbital integrals}

Now we return to our original setup, namely that of \Cref{Subsection:O-Witt vector twisted affine Springer fibers}.
Let $\kappa_\infty \colon X_*(\widetilde{G}_{\delta_\infty s})_{\Gamma_\infty, \tors} \to \mbC^{\times}$
be a character.
After enlarging $B$ if necessary, we have isomorphisms
$\psi_i \colon X_*(\widetilde{G}_{\delta_\infty s}) \simeq X_*(\widetilde{G}_{\delta_i s})$
for all $i \geq B$ as in \eqref{equation: psi for cocharacter groups}.
Let $\kappa_i \colon X_*(\widetilde{G}_{\delta_i s})_{\Gamma_i, \tors} \to \mbC^{\times}$
be the character corresponding to $\kappa_\infty$ via $\psi_i$.
As in \Cref{Proposition:algebra isomorphism of parahoric Hecke algebras}, we have isomorphisms
\[
\psi_i \colon \mcH_{\widetilde{J}_\infty}(\widetilde{G}(F_\infty)) \simlr \mcH_{\widetilde{J}_i}(\widetilde{G}(F_i))
\]
for all $i \in \mbN$.
Let $f_\infty \in \mcH_{\widetilde{J}_\infty}(\widetilde{G}(F_\infty))$ be a function,
and let $f_i := \psi_i(f_\infty) \in \mcH_{\widetilde{J}_i}(\widetilde{G}(F_i))$.

We can now prove the local constancy of twisted $\kappa$-orbital integrals in families of close local fields, which is our main technical result. 

\begin{Theorem}[Local constancy of twisted $\kappa$-orbital integrals]\label{Theorem: local constancy of kappa orbital integral}
    There exists an integer $B' \geq B$ such that for all $i \geq B'$, we have
    \[
    TO^{\kappa_i}_{\delta_i}(f_i) = TO^{\kappa_\infty}_{\delta_\infty}(f_\infty).
    \]
\end{Theorem}

\begin{proof}
    We will enlarge $B$ freely (often without explicitly mentioning it).
    It suffices to prove the assertion for $f_\infty=T_{w_\infty}$ for each $w \in \widetilde{W}(\widetilde{G}_F)$.
    Then $f_i = T_{w_i}$.
    By (the proof of) \Cref{Proposition: lattices in cocharacter groups}, we may assume that
    there exists a $\Gal(\overline{k}/k)$-equivariant isomorphism
    $\pi_1(\widetilde{G}_i) \simeq \pi_1(\widetilde{G}_\infty)$ which is compatible with
    $X_*(\widetilde{G}_{\delta_\infty s}) \simeq X_*(\widetilde{G}_{\delta_i s})$
    for all $i \geq B$.
    Then both $c_1$ and $c_2$ appearing in \Cref{Proposition:geometric interpretation of stable twisted orbital integrals} are independent of $i$.
    Moreover, we may naturally identify $\mcD(\Lambda_{\delta_\infty s}/F_\infty) \subset \mcD(\widetilde{G}_{\delta_\infty s}/F_\infty)$ with $\mcD(\Lambda_{\delta_i s}/F_i) \subset \mcD(\widetilde{G}_{\delta_i s}/F_i)$.
    Let $\{ t_{1, \infty}, \dotsc, t_{m, \infty} \} \subset \mcD(\widetilde{G}_{\delta_\infty s}/F_\infty)$ be a complete set of representatives for $\mcD(\widetilde{G}_{\delta_\infty s}/F_\infty)/\mcD(\Lambda_{\delta_\infty s}/F_\infty)$
    and let $\{ t_{1, i}, \dotsc, t_{m, i} \} \subset \mcD(\widetilde{G}_{\delta_i s}/F_i)$ be the corresponding elements.
    Let $\widetilde{t}_{j, \infty} \in \widetilde{G}_{\delta_\infty s}(\breve{F}_\infty)$ be a lift of $t_{j, \infty}$.
    Since $\mcO_{\Spa \breve{F}, \infty}$ is henselian, we may assume that $\widetilde{t}_{j, \infty}$ has a lift $\widetilde{t}_{j} \in \widetilde{G}_{\delta s}(\breve{F}_{\geq B})$.
    By \Cref{Corollary: local constancy of cocharacter groups}, we may assume that the fiber $\widetilde{t}_{j, i} \in \widetilde{G}_{\delta_i s}(\breve{F}_{i})$ is a lift of $t_{j, i}$.
    
    Let $X^{w}_{\widetilde{\mcP}, \delta s}$
    be
    the twisted affine Springer fiber 
    associated with the triple $(\widetilde{\mcP}, \delta, S^\circ_w)$.
    By \Cref{Theorem: finiteness of twisted affine Springer fiber}, the quotient $\underline{\Lambda}_{\delta s} \backslash (X^{w}_{\widetilde{\mcP}, \delta s})_{\mcO_{\geq B}/\pi}$ is a separated pfp perfect algebraic space over $\Spec \mcO_{\geq B}/\pi$.
    Applying \Cref{Proposition: category of pfp algebraic spaces} to $\underline{\Lambda}_{\delta s} \backslash (X^{w}_{\widetilde{\mcP}, \delta s})_{\mcO_{\geq B}/\pi}$ and its automorphism induced by $\widetilde{t}_{j}$, we obtain an isomorphism
    $\underline{\Lambda}_{\delta_{\infty} s} \backslash X^{w_\infty}_{\widetilde{\mcP}_\infty, \delta_\infty s} \simeq \underline{\Lambda}_{\delta_i s} \backslash X^{w_i}_{\widetilde{\mcP}_i, \delta_i s}$
    which is compatible with the actions of $\widetilde{t}_{j, \infty}$ and $\widetilde{t}_{j, i}$ for all $i \geq B$.
    In particular
    we have 
    \[
    (\underline{\Lambda}_{\delta_{\infty} s} \backslash X^{w_\infty}_{\widetilde{\mcP}_\infty, \delta_\infty s})(\overline{k})^{\widetilde{t}_{j, \infty}\sigma=1}\simeq 
    (\underline{\Lambda}_{\delta_i s} \backslash X^{w_i}_{\widetilde{\mcP}_i, \delta_i s})(\overline{k})^{\widetilde{t}_{j, i}\sigma=1}.
    \]
    Moreover, by \Cref{Proposition: pfp morphisms and limit preserving functors}, we may assume that the maps $x \mapsto \lambda_x$ appearing in \Cref{Proposition:geometric interpretation of stable twisted orbital integrals}
    for $i \geq B$ and $i=\infty$
    agree with each other under these identifications.
    By applying this argument to all
    $1 \leq j \leq m$, we can conclude from \Cref{Proposition:geometric interpretation of stable twisted orbital integrals} (where we may extend $\kappa_i$ to a character
    $X_*(\widetilde{G}_{\delta_i s})_{\Gamma_i} \to \mbC^{\times}$
    compatibly) that
    $
    TO^{\kappa_i}_{\delta_i}(T_{w_i}) = TO^{\kappa_\infty}_{\delta_\infty}(T_{w_\infty})
    $
    for all $i \geq B$ after enlarging $B$.
\end{proof}

\section{Base change fundamental lemma for local function fields}\label{Section: BCFL}

In this section, we apply the results obtained in the previous sections to prove the base change fundamental lemma for local function fields.

\subsection{Statement of the base change fundamental lemma}\label{Subsection:Statement of the base change fundamental lemma}
We state the base change fundamental lemma for central elements in parahoric Hecke algebras following \cite{HainesBCFLparahoric} (although the statement is formulated in characteristic zero in \textit{loc.cit.}, it remains valid in positive characteristic as well).
Note that our statement will be valid for a general local function field, but we will restrict ourselves to strongly regular semisimple conjugacy classes (see \Cref{Remark BCFL in char 0} for comments on this assumption).

We keep the notation from \Cref{Situation: unramified reductive group in families of close fields} and assume $k$ is finite.
At the moment, we work over $F_i$ for some $i\in \mathbb{N}\cup \lbrace \infty \rbrace$ and omit the subscript $i$ from the notation.
Let $\mcH_{I}(G(F))$ (resp.\ $\mcH_{J}(G(F))$) be the Iwahori (resp.\ parahoric) Hecke algebra as in \Cref{Subsection:Parahoric Hecke algebras in families of close local fields}.
Furthermore let $\widetilde{k}/k$ be a finite extension of degree $r$, and keep the notation from \Cref{Situation: twisted affine Springer fiber}.
Let $\mcH_{\widetilde{I}}(G(\widetilde{F}))$ (resp.\ $\mcH_{\widetilde{J}}(G(\widetilde{F}))$) be the corresponding Iwahori (resp.\ parahoric) Hecke algebra.

We give a quick recollection of the base change homomorphism
$
b\colon Z(\mcH_{\widetilde{J}}(G(\widetilde{F})))\rightarrow Z(\mcH_{J}(G(F)))
$
as defined by Haines \cite[Section 3]{HainesBCFLparahoric}.
Let $R=\mathbb{C}[X_{*}(S)]$
and let ${}_F W:=W(S)$ be the relative Weyl group.
Recall the Bernstein isomorphism (of $\mbC$-algebras)
\[
R^{{}_F W} \simlr Z(\mcH_{I}(G(F)))
\]
from \cite[Section 3.1]{HainesBCFLparahoric}.
It is obtained as the restriction of the map
$R \to \mcH_{I}(G(F))$
of $\mbC$-vector spaces
sending $\lambda\in X_{*}(S)$ to an element $\Theta_{\lambda}\in \mcH_{I}(G(F))$ that can be described as follows:
We may write $\lambda=\lambda_{1}-\lambda_{2},$ where $\lambda_{1},\lambda_{2}\in X_{*}(S)$ are $B$-dominant, and then
\begin{equation}\label{eq: Writing out Thetalambda in terms of the basis}
    \Theta_{\lambda}=\delta_{B}^{1/2}(\lambda(\pi_{F}))T_{\lambda_1(\pi_{F})} \cdot T_{\lambda_2(\pi_{F})}^{-1}.
\end{equation}
See \cite[Remark 1.7.2]{HainesKottwitzAPrasadIwahoriHecke}. 
Although \cite{HainesBCFLparahoric} assumes that $F$ has characteristic zero and \cite{HainesKottwitzAPrasadIwahoriHecke} also assumes that $G$ is split, the argument also works when $G$ is unramified and $F$ is of positive characteristic; see also \cite{RostamiBernsteinGeneralReductive} and \cite[Appendix]{HainesBernsteincenter}.
Haines proves in \cite[Theorem 3.1.1]{HainesBCFLparahoric} that convolution with the characteristic function $1_{J}$ gives an isomorphism of $\mathbb{C}$-algebras
$
-\cdot 1_{J}\colon Z(\mcH_{I}(G(F)))\simeq Z(\mcH_{J}(G(F))).
$
Composing with the Bernstein isomorphism we obtain the following isomorphism, which we also call the Bernstein isomorphism:
\[
R^{{}_F W} \simlr Z(\mcH_{J}(G(F))).
\]
Let $S^{\widetilde{F}}$ be the maximal $\widetilde{F}$-split subtorus of $T$ which is defined over $F$
and let ${}_{\widetilde{F}} W=W(S^{\widetilde{F}})$ be the corresponding relative Weyl group.
Observe that we obtain a norm map
$$
N\colon \mathbb{C}[X_{*}(S^{\widetilde{F}})]^{{}_{\widetilde{F}}W}\rightarrow \mathbb{C}[X_{*}(S)]^{{}_{F}W}
$$
induced by
$
\lambda\mapsto \sum_{i=0}^{r-1}\sigma^{i}(\lambda).
$
\begin{Definition}[Base change homomorphism]\label{Def: base change homomorphism}
    The base change homomorphism
    \[
    b\colon Z(\mcH_{\widetilde{J}}(G(\widetilde{F})))\rightarrow Z(\mcH_{J}(G(F)))
    \]
    is the unique $\mathbb{C}$-algebra homomorphism making the following diagram commute:
    $$
    \xymatrix{
    \mathbb{C}[X_{*}(S^{\widetilde{F}})]^{{}_{\widetilde{F}}W} \ar[r]^{} \ar[d]^{N} & Z(\mcH_{\widetilde{J}}(G(\widetilde{F}))) \ar[d]^{b} \\
    \mathbb{C}[X_{*}(S)]^{{}_{F}W} \ar[r]^{} & Z(\mcH_{J}(G(F))),
    }
    $$
    where the horizontal maps are the Bernstein isomorphisms.
\end{Definition}

The base change fundamental lemma for local function fields is stated as follows:

\begin{Theorem}[Base change fundamental lemma]\label{Thm: base change fundamental lemma}
    We assume that $F=k((\pi_\infty))$.
    Let $f\in Z(\mcH_{\widetilde{J}}(G(\widetilde{F})))$ with base change $b(f)\in Z(\mcH_{J}(G(F))).$
    Let $\gamma\in G(F)$
    be a strongly regular semisimple element which is $G(F^{\sep})$-conjugate to
    $N(\delta)=\delta \cdot \sigma(\delta)\cdots \sigma^{r-1}(\delta) \in G(\widetilde{F})$
    for some $\delta\in G(\widetilde{F})$.
    Then we have the equality
    \[
    STO_{\delta}(f)=SO_{\gamma}(b(f)).
    \]
\end{Theorem}

\begin{Remark}\label{Remark BCFL in char 0}
    If $F$ is of characteristic zero, then \Cref{Thm: base change fundamental lemma} is proved in \cite{HainesBCFLparahoric}.
    In fact, in \textit{loc.cit.}, it is formulated and proved not only for strongly regular semisimple elements but also for general semisimple elements.
    Our result extends immediately to regular semisimple elements using $z$-extensions but we expect that additional work would be necessary in positive characteristic to further reduce to general semisimple elements.
    Moreover, in \textit{loc.cit.}, it is also proved that if $\gamma$ is not a norm, then $SO_{\gamma}(b(f))=0$.
    One might expect that the vanishing of $SO_{\gamma}(b(f))$ when $\gamma$ is not a norm is approachable via the technique of close fields.
    However, in characteristic zero, this vanishing assertion is proved by purely local arguments (following Labesse \cite[Proposition 3.7.2]{LabesseAsterisque}).
    Since these arguments are already available in positive characteristic, we will not pursue this direction further here.
\end{Remark}
\subsection{Proof of \Cref{Thm: base change fundamental lemma}}\label{Proof of BCFL}

Let $F=\lbrace (F_{i},\pi_{i}) \rbrace_{i\in \mathbb{N} \cup \lbrace \infty \rbrace}$ be a family of close local fields as in \Cref{Subsection:Unramified reductive groups in families}.
We first note the following compatibility between Bernstein isomorphisms and a family of close local fields.

\begin{Remark}\label{Remark:compatibility of base change map in close local fields}
    By \Cref{Proposition:algebra isomorphism of parahoric Hecke algebras}, we have isomorphisms of \textit{$\mbC$-algebras}
\[
\psi_i \colon \mcH_{J_\infty}(G(F_\infty)) \simlr \mcH_{J_i}(G(F_i)),
\]
which induce $Z(\mcH_{J_\infty}(G(F_\infty))) \simlr Z(\mcH_{J_i}(G(F_i)))$.
Moreover, we have a natural identification
\[
\mathbb{C}[X_{*}(S_\infty)]^{{}_{F_\infty}W} \simlr \mathbb{C}[X_{*}(S_i)]^{{}_{F_i}W}.
\]
By the formula \eqref{eq: Writing out Thetalambda in terms of the basis}, Bernstein isomorphisms are compatible with these identifications (in an obvious sense).
In the hyperspecial case, Bernstein isomorphisms agree with Satake isomorphisms, and thus we also see that Satake isomorphisms are compatible with these identifications.
\end{Remark}

We also need the following lemma.

\begin{Lemma}\label{lem: spreading out norms}
    Let $\gamma_{\infty}\in G(F_{\infty})$
    be a strongly regular semisimple element
    which is $G(F^{\sep}_\infty)$-conjugate to
    $N(\delta_\infty)$
    for some $\delta_\infty \in G(\widetilde{F}_\infty)$.
    Then there exist $B$ and lifts $\gamma_{\geq B}\in G(F_{\geq B})$
    and
    $\delta_{\geq B}\in G(\widetilde{F}_{\geq B})$
    of $\gamma_\infty$ and $\delta_\infty$, respectively, such that for all $i \geq B$,
    $\gamma_{i}$ is $G(F^{\sep}_i)$-conjugate to
    $N(\delta_i)$.
\end{Lemma}
\begin{proof}
    By \Cref{Proposition: strongly regular-semi-simple spreads out}, we find $B$ and a lift
    $\delta \in G(\widetilde{F}_{\geq B})$ of $\delta_\infty$, and 
    we may assume that the centralizer $G_{N(\delta)}$ of 
    $N(\delta) \in G(\widetilde{F}_{\geq B})$
    is a maximal torus.
    Then the conjugacy class $G_{\widetilde{F}_{\geq B}}/G_{N(\delta)}$
    of $N(\delta)$
    forms a smooth quasi-affine scheme over $\Spec \widetilde{F}_{\geq B}$
    (see \cite[Theorem 2.3.1]{ConradReductiveGroupSchemes} for example).
    It descends to a smooth quasi-affine scheme $\mcC$ over $\Spec F_{\geq B}$
    along the $\Gal(\widetilde{k}/k)$-torsor 
    $\Spec \widetilde{F}_{\geq B} \to \Spec F_{\geq B}$.
    By assumption, we have the $F_{\infty}$-valued point $\gamma_{\infty}\in \mcC(F_\infty)$ and by smoothness, we find after increasing $B$ a lift $\gamma \in \mcC(F_{\geq B})$ of $\gamma_\infty$.
    By construction
    $\gamma_{i}$ is $G(F^{\sep}_i)$-conjugate to
    $N(\delta_i)$ for all $i \geq B$.
\end{proof}

We now give the proof of Theorem \ref{Thm: base change fundamental lemma}.

\begin{proof}
Let $\gamma_{\infty}\in G(F_{\infty})$
    be a strongly regular semisimple element
    which is $G(F^{\sep}_\infty)$-conjugate to
    $N(\delta_\infty)$
    for some $\delta_\infty \in G(\widetilde{F}_\infty)$.
    Let $f_\infty \in Z(\mcH_{\widetilde{J}_\infty}(G(\widetilde{F}_\infty)))$. 
    As in \Cref{Proposition:algebra isomorphism of parahoric Hecke algebras}, we have isomorphisms
    $
\psi_i \colon \mcH_{\widetilde{J}_\infty}(G(\widetilde{F}_\infty)) \simlr \mcH_{\widetilde{J}_i}(G(\widetilde{F}_i))
$
    of $\mbC$-algebras for all $i \in \mbN$.
Let $f_i := \psi_i(f_\infty) \in Z(\mcH_{\widetilde{J}_i}(G(\widetilde{F}_i)))$.
Similarly, we have 
$
\psi_i \colon \mcH_{J_\infty}(G(F_\infty)) \simlr \mcH_{J_i}(G(F_i)).
$
Then we have
$\psi_i(b(f_\infty))=b(f_i)$
by \Cref{Remark:compatibility of base change map in close local fields}.
Let $\gamma_{\geq B}\in G(F_{\geq B})$
    and
    $\delta_{\geq B}\in G(\widetilde{F}_{\geq B})$
    be lifts of $\gamma_\infty$ and $\delta_\infty$, respectively, as constructed in \Cref{lem: spreading out norms}.
Then, by Theorem \ref{Theorem: local constancy of kappa orbital integral}, there exists $B' \geq B$ such that
$SO_{\gamma_i}(b(f_i)) = SO_{\gamma_\infty}(b(f_\infty))$
and $STO_{\delta_i}(f_i) = STO_{\delta_\infty}(f_\infty)$
for all $i \geq B'$.
Thus, the desired result follows from the known result \cite{HainesBCFLparahoric} in characteristic zero.
\end{proof}

\section{Close local fields and endoscopy}\label{Section: Close fields and Endoscopy}

In this final section, we prove the standard endoscopic fundamental lemma in positive characteristic.
For this purpose, we show that
the Langlands--Shelstad transfer factors in the unramified endoscopic setting are locally constant for families of close local fields (see \Cref{proposition: local constancy of transfer factors} for the precise statement).
By combining this local constancy with our results from previous sections, we complete the proof of the fundamental lemma by reducing the assertion to the case of characteristic zero.

\subsection{Recollections on endoscopy and the standard endoscopic fundamental lemma}\label{Subsection: Recollections on endoscopy and the standard endoscopic fundamental lemma}

We briefly recall the parts of the theory of endoscopy that are needed for our situation at hand; since we are working with unramified groups, the discussion simplifies considerably compared to the general case.
We follow Hales' definition of an unramified endoscopic datum as given in \cite{HalesSimpleDefinitionTransferUnramified} or \cite{HalesReductionToUnit}, which itself is based on Langlands--Shelstad's definition in \cite{LanglandsShelstadTransfer} (see also \cite[Appendix A]{GriffinWangBook}).

In this subsection, let $F$ be a non-archimedean local field with finite residue field $k$.
We consider a pinned reductive group scheme
$(\mcG, \mcT, \mcB, \{ x_\alpha \}_{\alpha \in \Delta})$
over $W(k)$
as in \Cref{Situation: unramified reductive group in families of close fields}.
Let $G:=\mcG_F$,
let $
(\hat{G}, \hat{T}, \hat{B})$
be the Langlands dual group over $\mbC$ with the pinning corresponding to $(\mcG, \mcT, \mcB, \{ x_\alpha \}_{\alpha \in \Delta})$, and let
${}^LG :=\hat{G}\rtimes W_{F}$
be the $L$-group.
Here $W_{F} \subset \Gamma=\Gal(F^{\sep}/F)$ is the Weil group of $F$.

\begin{Definition}\label{Definition: unramified endoscopic datum}
    An unramified endoscopic datum for $G$ is a triple $(\mcH,s,\xi),$ where
    \begin{enumerate}
        \item $\mcH$ is a reductive group scheme over $\mcO_F$ with a $\Gamma$-stable pinning, with $H:=\mathcal{H}_F$ and ${}^L H = \hat{H} \rtimes W_F$,
        \item $s \in \hat{G}$ is a semisimple element, and
        \item $\xi\colon {}^L H\hookrightarrow {}^L G$ is a homomorphism of extensions of $W_{F},$
    \end{enumerate}
    satisfying the following conditions:
        \begin{enumerate}
            \item[(i)] $\xi$ induces an isomorphism $\hat{H}\simeq \Cent_{\hat{G}}(s)^{\circ}.$
            \item[(ii)] There exists a $1$-cocycle $b$ of $W_F$ in $Z(\hat{G})$ which is trivial in $H^{1}(W_{F},Z(\hat{G}))$ such that
            $
            s^{-1}\xi(x,w)s=b(w)\xi(x,w).
            $
            \item[(iii)] $\xi$ factors through a map $\hat{H}\rtimes \Gal(E/F)\rightarrow \hat{G}\rtimes \Gal(E/F)$ for 
            a finite unramified extension $E/F$.
        \end{enumerate}
\end{Definition}

\begin{Remark}\label{Remark:simplified conditions}
    We identify $\hat{H}$ with $\Cent_{\hat{G}}(s)^{\circ}$ via $\xi$.
Up to equivalence (in the sense of \cite{LanglandsShelstadTransfer}), we may assume that
$s \in \hat{H}$ is invariant with respect to the action of $\Gamma$ on $\hat{H}$.
We may also assume that
the maximal torus
of
$\hat{H}$
associated with the pinning of $H$
is $\hat{T}$
and the Borel subgroup
of $\hat{H}$
associated with the pinning of $H$, which we denote by $\hat{B}_H$, is contained in $\hat{B}$; in particular $s \in \hat{T}$.
\end{Remark}

We recall the following construction from \cite[Section 1.3]{LanglandsShelstadTransfer}.
Here we will use \cite[Corollary 2.2]{KottwitzRationalConjugacy} in our situation where $F$ is not necessarily perfect.
Although Kottwitz uses \cite[Theorem 9.8]{SteinbergRegSS} (in the proof of \cite[Lemma 2.1]{KottwitzRationalConjugacy}), the needed statement for regular semisimple conjugacy classes over imperfect fields is proved in \cite[Section 8.6]{Borel-Springer}, which suffices for our purposes.

\begin{Construction}\label{Construction: transfering strongly regular semisimple elements endoscopic groups}
Let $(\mcH,s,\xi)$ be an unramified endoscopic datum for $G$.
Let $\gamma_H \in H(F)$ be a strongly regular semisimple element.
Choose a Borel subgroup $B'_{H}\subset H_{F^{\sep}}$ containing $(H_{\gamma_H})_{F^{\sep}}$.
This, together with $\hat{B}_H$, induces a canonical isomorphism
$\hat{T} \simeq \widehat{H_{\gamma_H}}$ with dual
$(H_{\gamma_H})_{F^{\sep}}\simeq T_{F^{\sep}}$.
By \cite[Corollary 2.2]{KottwitzRationalConjugacy}, 
a conjugate $\iota$
of
the composite
$(H_{\gamma_H})_{F^{\sep}} \simeq T_{F^{\sep}} \hookrightarrow G_{F^{\sep}}$
by some element $h\in G(F^{\sep})$ is defined over $F$.
Such an embedding $\iota \colon H_{\gamma_H} \hookrightarrow G$ over $F$ will be called an \textit{admissible embedding}.
Let $\gamma := \iota(\gamma_H) \in G(F)$.
We call $\gamma_H$ \textit{strongly $G$-regular semisimple} if $\gamma$ is strongly regular semisimple.
(This notion is independent of the choice of $\iota$.)
Then $\iota$ induces an isomorphism $H_{\gamma_H} \simlr G_\gamma$.
\end{Construction}

We fix a strongly $G$-regular semisimple element $\gamma_{H} \in H(F)$ 
and an admissible embedding $\iota \colon H_{\gamma_H} \hookrightarrow G$ as in \Cref{Construction: transfering strongly regular semisimple elements endoscopic groups}.
Let $\gamma:=\iota(\gamma_H)$.
Then we have $s\in (\widehat{H_{\gamma_H}})^{\Gamma}$ and it induces an element of $(\widehat{G_{\gamma}})^{\Gamma}$,
whose image in $\pi_{0}((\widehat{G_{\gamma}})^{\Gamma})$
is identified with a character $\kappa\colon H^{1}(F,G_{\gamma})=X_{*}(G_{\gamma})_{\Gamma, \tors}\rightarrow \mathbb{C}^{*}$ by Tate--Nakayama duality.

Let $\Delta_{0}(\gamma_{H},\gamma)$ denote the transfer factor in \cite{LanglandsShelstadTransfer} (see also \Cref{Rem: Transfer factors after endo FL}).
Furthermore, we denote by 
\[
b_{\xi}\colon \mcH_{\sph}(G(F))\rightarrow \mcH_{\sph}(H(F))
\]
the $\mathbb{C}$-algebra homomorphism defined via $\xi\colon {}^LH\rightarrow {}^LG$ and the Satake isomorphisms, where $\mcH_{\sph}(G(F))$ and $\mcH_{\sph}(H(F))$
are the spherical Hecke algebras associated with $\mcG(\mcO_F)$ and $\mcH(\mcO_F)$, respectively.
The fundamental lemma for standard endoscopy in positive characteristic is the following statement.

\begin{Theorem}[Standard endoscopic fundamental lemma]\label{Theorem: Standard Endo FL}
We assume that $F$ is of positive characteristic.
    For any $f\in \mcH_{\sph}(G(F))$, the following equality holds
    \begin{equation}\label{eq: identity standard endo FL}
    \Delta_{0}(\gamma_{H},\gamma)O^{\kappa}_{\gamma}(f)=SO_{\gamma_{H}}(b_{\xi}(f)).
    \end{equation}
\end{Theorem}

\begin{Remark}\label{Rem: Transfer factors after endo FL}
    Let $\gamma_G \in G(F)$ be an element which is $G(F^{\sep})$-conjugate to $\gamma$.
    Under the choices of the admissible embedding $\iota$, $a$-data and $\chi$-data (in the sense of \cite{LanglandsShelstadTransfer}),
    we define
    the transfer factor
    \[
    \Delta_0(\gamma_H, \gamma_G)=\Delta_I(\gamma_H, \gamma_G)\Delta_{II}(\gamma_H, \gamma_G)\Delta_{III_1}(\gamma_H, \gamma_G)\Delta_{III_2}(\gamma_H, \gamma_G)\Delta_{IV}(\gamma_H, \gamma_G)
    \]
    in the same way as in \cite{LanglandsShelstadTransfer}.
    Although $F$ is assumed to be of characteristic zero in \textit{loc.cit.}, the same constructions work; see \cite[Appendix A]{GriffinWangBook}.
    One can show that $\Delta_0(\gamma_H, \gamma_G)$ is independent of the choices of $\iota$, $a$-data and $\chi$-data.
    The left-hand side of
    \eqref{eq: identity standard endo FL} is equal to 
    \[
    \sum_{\gamma_G} \Delta_{0}(\gamma_{H},\gamma_G) O_{\gamma_G}(f)
    \]
    where $\gamma_G \in G(F)$ runs over a complete set of representatives of conjugacy classes inside the $G(F^{\sep})$-conjugacy class of $\gamma$, and in particular it is independent of $\iota$.
\end{Remark}

\subsection{Proof of Theorem \ref{Theorem: Standard Endo FL}}\label{Subsection: Proof of Theorem Standard Endo FL}

Let $F_\infty$ be a non-archimedean 
local field of positive characteristic
with finite residue field $k$.
We shall prove \Cref{Theorem: Standard Endo FL} for $F_\infty$.
We choose
a family
$F=\lbrace (F_{i},\pi_{i}) \rbrace_{i\in \mathbb{N} \cup \lbrace \infty \rbrace}$
of close local fields
approximating $F_\infty$, and we use the notation of \Cref{Section:Reductive groups in families of close fields}.
Let $(\mcH_\infty, s, \xi_\infty)$
be an unramified endoscopic datum for $G_\infty$.
Without loss of generality, we may assume that $(\mcH_\infty, s, \xi_\infty)$ satisfies the conditions given in \Cref{Remark:simplified conditions}.
There exists a reductive group scheme $\mcH$ over $W(k)$ with a $\Gal(\overline{k}/k)$-stable pinning such that its base change to $\mcO_\infty$ is $\mcH_\infty$.
Let $\mcH_i := \mcH_{\mcO_i}$
and $H_i := \mcH_{F_i}$ for $i \in \mbN \cup \{ \infty \}$.
Note that for all $i$, we have
a canonical identification $\hat{H}_{i}=\hat{H}_{\infty}$, and
$\xi_\infty \colon {}^L H_\infty\hookrightarrow {}^L G_\infty$ 
naturally induces a homomorphism
$\xi_i \colon {}^L H_i\hookrightarrow {}^L G_i$
of extensions of $W_{F_i}$
such that 
$(\mcH_i, s, \xi_i)$
is an unramified endoscopic datum for $G_i$.

Let $H:=\mcH \times_{\Spec W(k)} \Spec F$.
Let $\gamma_{H_{\infty}}\in H(F_{\infty})$ be a strongly $G$-regular semisimple element.
We may spread out $\gamma_{H_{\infty}}$ by \Cref{Proposition: strongly regular-semi-simple spreads out} to an element $\gamma_{H_{\geq B}} \in H(F_{\geq B})$ for some $B \geq 1$ that is fiberwise strongly regular semisimple.
For the torus $H_{\gamma_H}$ over $F_{\geq B}$,
we assume that we are in \Cref{Situation: torus} and use the notation there.
In particular $E=\{ (E_i, \pi_{E_i}) \}_{i \in \mbN \cup \{ \infty \}}$ is a family of close local fields over $F$, and $H_{\gamma_H}$ is split over $E_{\geq B}$.
We may further assume that 
there exists a Borel subgroup $B'_H \subset H_{E_{\geq B}}$ containing $H_{\gamma_H}$.
As in \Cref{Construction: transfering strongly regular semisimple elements endoscopic groups}, this induces
a canonical isomorphism
$(H_{\gamma_H})_{E_{\geq B}} \simeq T_{E_{\geq B}}$.

\begin{Lemma}\label{Lemma: matching pair in families of close local fields}
    Let $T_H$ be an $F_{\geq B}$-torus which is split over $E_{\geq B}$, and let
    $\iota \colon (T_H)_{E_{\geq B}} \hookrightarrow G_{E_{\geq B}}$
    be
    a closed embedding over $E_{\geq B}$ such that the image of $\iota$ is a maximal torus of $G_{E_{\geq B}}$.
    We assume that some conjugate of $\iota_\infty \colon (T_H)_{E_{\infty}} \to G_{E_{\infty}}$ under $G(E_\infty)$ is defined over $F_\infty$.
    Then some conjugate of $\iota$ under $G(E_{\geq B})$ is defined over $F_{\geq B}$ after enlarging $E$ and $B$.
\end{Lemma}

\begin{proof}
    We may assume that $\iota((T_H)_{E_{\geq B}}) \subset G_{E_{\geq B}}$ descends to a maximal $F_{\geq B}$-torus $T' \subset G_{F_{\geq B}}$.
    Let $h_\infty \in G(E_\infty)$ be such that $\ad(h_\infty) \circ \iota_\infty$ is defined over $F_\infty$.
    Then, for any element $g \in \Gal(E_\infty/F_\infty)$,
    we have
    $h^{-1}_\infty g(h_\infty) \in N_G(T')(E_\infty)$.
    Let $W(T')$ be the Weyl group scheme over $\Spec F_{\geq B}$ associated with $T'$.
    Let $w_{g, \infty} \in W(T')(E_\infty)$ be the image of $h^{-1}_\infty g(h_\infty)$.
    Then $g \mapsto w_{g, \infty}$ is a $1$-cocycle.
    Since $W(T')$ is finite and \'etale over $\Spec F_{\geq B}$, we have $W(T')(\mcO_{\Spa E, \infty})=W(T')(E_\infty)$.
    So, after increasing $B$, we can extend the $1$-cocycle to a $1$-cocycle
    $g \mapsto w_g$
    of $\Gal(E_\infty/F_\infty)$
    in $W(T')(E_{\geq B})$
    such that
    $\ad(w_g) = \iota \circ ({}^g\iota)^{-1}$.
    The quotient $G_{F_{\geq B}}/T'$ is a smooth quasi-affine scheme over $F_{\geq B}$ (see \cite[Theorem 2.3.1]{ConradReductiveGroupSchemes}).
    We twist $G_{F_{\geq B}}/T'$ by $g \mapsto w^{-1}_g$ and obtain a smooth quasi-affine scheme ${}^*(G_{F_{\geq B}}/T')$ over $F_{\geq B}$
    as in \cite[Section 2.1]{KottwitzRationalConjugacy}.
    Then $h_\infty$ gives rise to an element in ${}^*(G_{F_{\geq B}}/T')(F_\infty)$.
    Since ${}^*(G_{F_{\geq B}}/T')$ is smooth, this element extends to an element $\overline{h} \in {}^*(G_{F_{\geq B}}/T')(F_{\geq B})$.
    Since $T'$ is split over $E_{\geq B}$, we can find an element $h \in G(E_{\geq B})$ which induces $\overline{h}$.
    Then one can check that $\ad(h) \circ \iota$ is defined over $F_{\geq B}$ as in the proof of \cite[Corollary 2.2]{KottwitzRationalConjugacy}.
\end{proof}

By \Cref{Lemma: matching pair in families of close local fields}, 
after enlarging $B$,
there exists an element $h_0 \in G(E_{\geq B})$
such that the conjugate of
$(H_{\gamma_H})_{E_{\geq B}} \simeq T_{E_{\geq B}} \hookrightarrow G_{E_{\geq B}}$
by $h_0$ is defined over $F_{\geq B}$.
We denote this conjugate
by 
$\iota \colon H_{\gamma_H} \hookrightarrow G_{F_{\geq B}}$.
Let $\gamma:=\iota(\gamma_H) \in G(F_{\geq B})$.
By \Cref{Proposition: strongly regular-semi-simple spreads out}, we may assume that $\gamma$ is fiberwise strongly regular semisimple.
Then we have
$
\iota \colon H_{\gamma_H} \simlr G_\gamma
$
over $F_{\geq B}$.
We note that $\iota$ fiberwise gives an admissible embedding in the sense of \Cref{Construction: transfering strongly regular semisimple elements endoscopic groups}.
It suffices to prove \eqref{eq: identity standard endo FL} for $\gamma_{H_\infty}$ and $\gamma_\infty$
(see \Cref{Rem: Transfer factors after endo FL}).
Note that the statement is known to be true
for non-archimedean local fields of characteristic zero in all residue characteristics by the work of Ngô \cite{NgoFL}, Waldspurger \cite{WaldspurgerEndoscopyChangeofChar} (see also Cluckers--Hales--Loeser \cite{CluckersHalesLoeser}), \cite{Waldspurgertorduenestpassitordue} and Hales \cite{HalesReductionToUnit}.
Therefore, it suffices to show that 
both sides of the equality \eqref{eq: identity standard endo FL} are locally constant around $\infty$.
This follows from \Cref{Theorem: local constancy of kappa orbital integral}
and \Cref{Remark:compatibility of base change map in close local fields}
together with the following proposition:

\begin{proposition}[Local constancy of transfer factors]\label{proposition: local constancy of transfer factors}
    After possibly enlarging $B,$ we have for all $i\geq B$
    $$
    \Delta_{0}(\gamma_{H_{i}},\gamma_{i})=\Delta_{0}(\gamma_{H_{\infty}},\gamma_{\infty}).
    $$
\end{proposition}

The rest of this subsection is devoted to the proof of \Cref{proposition: local constancy of transfer factors}.
We will enlarge $B$ freely (often without explicitly mentioning it). 
By construction, $B'_G:= h_0(B_{E_{\geq B}})h^{-1}_0 \subset G_{E_{\geq B}}$ is a Borel subgroup containing $(G_\gamma)_{E_{\geq B}}=h_0(T_{E_{\geq B}})h^{-1}_0$.
We will prove the proposition by reviewing some constructions of Langlands--Shelstad.
We will perform these constructions with respect to 
the pinnings of $\mcG$ and $\mcH$
and
the Borel subgroups $B'_H$ and $B'_G$.
Using $\iota$ and a fixed choice of $a$-data and $\chi$-data, the transfer factor for $i \in \mbN \cup \{ \infty \}$
can be written as
\[
\Delta_{0}(\gamma_{H_{i}},\gamma_{i})=\Delta_I(\gamma_{H_i}, \gamma_i)\Delta_{II}(\gamma_{H_i}, \gamma_i)\Delta_{III_2}(\gamma_{H_i}, \gamma_i)\Delta_{IV}(\gamma_{H_i}, \gamma_i)
\]
since $\Delta_{III_1}(\gamma_{H_i}, \gamma_i)=1$.
We make the following choices for the $a$-data and $\chi$-data; with these choices, our goal is to prove that each term appearing in the product is locally constant around $\infty$.
To simplify the notation, we write
\[
T_H:=H_{\gamma_H}
\quad \text{and} \quad
T_G:=G_{\gamma}.
\]
We may assume that $T_G$ is also split over $E_{\geq B}$, and by taking the dual of \eqref{equation: psi for cocharacter groups},
we have identifications
$X^*(T_{G, i}) \simeq X^*(T_{G, \infty})$
and
$X^*(T_{H, i}) \simeq X^*(T_{H, \infty})$
which are equivariant with respect to the actions of $\Gamma_i/I^n_i \simeq \Gamma_\infty/I^n_\infty$ for some $n \leq B$.
We may further assume that these identifications induce
$R(G_i, T_{G, i}) \simeq R(G_\infty, T_{G, \infty})$
and
$R(H_i, T_{H, i}) \simeq R(H_\infty, T_{H, \infty})$
for the sets of absolute roots.
For ease of notation, we will simply denote these sets by $R_G$ and $R_H$, respectively.
We identify 
$R_H$ as a subset of $R_G$ using $\iota$.

\begin{Construction}[$a$-data]\label{Construction:a-data}
    We can choose
    a subset
    $\{ a_{\alpha} \}_{\alpha \in R_G} \subset E^{\times}_{\geq B}$
    such that $a_{g \alpha}=g(a_{\alpha})$ for all $g\in \Gal(E_\infty/F_\infty)$ and
    $a_{-\alpha}=-a_{\alpha}.$
    On each fiber $E_i$ for $B \leq i \leq \infty$,
    the set $\{ a_{\alpha, i} \}_{\alpha \in R_G} \subset E^\times_i$
    gives a set of $a$-data for $T_{G, i}$ in the sense of \cite[(2.2)]{LanglandsShelstadTransfer}, which we will use to compute the transfer factor.
\end{Construction}

For $\alpha \in R_G$ and each $B \leq i \leq \infty$, let
$F_i \subset F_{\alpha, i} \subset E_i$
be the subfield corresponding to the stabilizer of $\alpha$ for the action of $\Gal(E_i/F_i) \simeq \Gal(E_\infty/F_\infty)$ on $R_G$.

\begin{Construction}[$\chi$-data]\label{Construction:chi-data}
We choose a set of $\chi$-data
$\{ \chi_{\alpha, \infty} \colon F^\times_{\alpha, \infty} \to \mbC^\times \}_{\alpha \in R_G}$
for $T_{G, \infty}$
in the sense of \cite[(2.5)]{LanglandsShelstadTransfer}
(see also \cite[Definition A.2.13]{GriffinWangBook}).
By local class field theory, we may regard
$\chi_{\alpha, \infty}$
as a continuous character $\chi_{\alpha, \infty} \colon W_{F_{\alpha, \infty}} \to \mbC^\times$.
We may assume that
the characters $\chi_\alpha$ factor through
$W_{F_{\alpha, \infty}}/I^n_\infty$.
After enlarging $B$, 
we can use identifications
$W_{F_{\alpha, i}}/I^n_i \simeq W_{F_{\alpha, \infty}}/I^n_\infty$
to transport $\{ \chi_{\alpha, \infty} \}_{\alpha \in R_G}$ to a set of $\chi$-data
$\{ \chi_{\alpha, i} \colon F^\times_{\alpha, i} \to \mbC^\times\}_{\alpha \in R_G}$
for $T_{G, i}$ for all $i \geq B$.
\end{Construction}

\begin{Lemma}\label{Lemma:local constancy of DeltaI}
$
    \Delta_{I}(\gamma_{H_{i}},\gamma_{i})=\Delta_{I}(\gamma_{H_{\infty}},\gamma_{\infty})
    $
    for all $i\geq B$
after possibly enlarging $B$.
\end{Lemma}

\begin{proof}
Since, by definition (see \cite[(3.2)]{LanglandsShelstadTransfer}), both sides coincide with those obtained by replacing $\mcG$ with $\mcG_{\mathrm{sc}}$, we may assume that $\mcG=\mcG_{\mathrm{sc}}$.
Recall that $k'$ is the residue field of $E_\infty$.
We may assume that the maximal torus $\mcT \subset \mcG$ over $W(k)$ is split over $W(k')$
and the pinning $\{ x_\alpha \}_{\alpha \in \Delta}$ of $\mcG$ is defined over $W(k')$.
Let $W(\mcT)$ be the Weyl group scheme over $W(k)$ associated with $\mcT$.
Let
\[
n \colon W(\mcT)(W(k')) \to N_{\mcG}(\mcT)(W(k'))
\]
be the Tits section associated with the pinning of $\mcG$
as in \cite[(2.1)]{LanglandsShelstadTransfer}.
Since our pinning is defined over $W(k')$, the map $n$ is actually defined over $W(k')$; see also \cite[Section A.3.1]{GriffinWangBook}.

For an element $g \in \Gal(E_\infty/F_\infty)$,
let $w_g \in W(\mcT)(E_{\geq B})$ be the image of $h^{-1}_0g(h_0)$.
Then $g \mapsto w_g$ is a $1$-cocycle of $\Gal(E_\infty/F_\infty)$ in $W(\mcT)(E_{\geq B})$.
Since $W(\mcT)$ is constant over $W(k')$,
by \Cref{Corollary: open nbd of infinity},
we may assume after enlarging $B$ that this $1$-cocycle
comes from a $1$-cocycle of $\Gal(E_\infty/F_\infty)$ in $W(\mcT)(W(k'))$, which we also denote by $g \mapsto w_g$.
Furthermore, our choice of $a$-data in \Cref{Construction:a-data} allows us to attach to each $g \in \Gal(E_\infty/F_\infty)$ an element $x_g \in T(E_{\geq B})$ whose restriction to each fiber $E_i$ ($B \leq i \leq \infty$) coincides with the element $x(w_g \rtimes g) \in T(E_i)$ constructed in \cite[(2.3)]{LanglandsShelstadTransfer} using
the set of $a$-data
$\{ a_{\alpha, i} \}_{\alpha \in R_G}$.
Then
one can show that
$g \mapsto h_0 x_gn(w_g)g(h_0)^{-1}$
is a $1$-cocycle of $\Gal(E_\infty/F_\infty)$ in $T_G(E_{\geq B})$.
Let $\lambda(T_G)$ denote its class in
$H^1(\Gal(E_\infty/F_\infty), T_G(E_{\geq B}))$
and let $\lambda(T_G)_i \in H^1(F_i, T_{G, i})$
be the image of $\lambda(T_G)$.
By definition,
we have
$\Delta_{I}(\gamma_{H_{i}},\gamma_{i})=\kappa_i(\lambda(T_G)_i)$ for the character $\kappa_i \colon H^1(F_i, T_{G, i}) \to \mbC^\times$ (corresponding to $s$).
The desired claim now follows from \Cref{Corollary:Tate-Nakayama compatibility}.
\end{proof}

\begin{Lemma}\label{Lemma:local constancy of DeltaII}
$
    \Delta_{II}(\gamma_{H_{i}},\gamma_{i})=\Delta_{II}(\gamma_{H_{\infty}},\gamma_{\infty})
    $
    for all $i\geq B$
after possibly enlarging $B$.
\end{Lemma}

\begin{proof}
    Recall that $\Delta_{II}(\gamma_{H_{i}},\gamma_{i})=\prod_{\alpha} \chi_{\alpha,i}\left(\frac{\alpha(\gamma_{i})-1}{a_{\alpha,i}}\right),
    $
    where $\alpha$ runs over a complete set of representatives of the $\Gal(E_\infty/F_\infty)$-orbits in $R_G \backslash R_H$.
    Thus, the assertion immediately follows from the constructions of $\chi_{\alpha, i}$ and $a_{\alpha,i}$, together with the compatibility between reciprocity maps in local class field theory and families of close local fields, as proved in \cite[Proposition 3.6.1]{Deligne84}.
\end{proof}

\begin{Lemma}\label{Lemma:local constancy of DeltaIII2}
$
    \Delta_{III_2}(\gamma_{H_{i}},\gamma_{i})=\Delta_{III_2}(\gamma_{H_{\infty}},\gamma_{\infty})
    $
    for all $i\geq B$
after possibly enlarging $B$.
\end{Lemma}

\begin{proof}
    Following \cite[(2.6)]{LanglandsShelstadTransfer}, the set of $\chi$-data $\{ \chi_{\alpha, i} \}_{\alpha \in R_G}$ induces an admissible embedding ${}^LT_{G, i} \to {}^LG_i$.
    Furthermore, its restriction provides a set of $\chi$-data for $R_H$, yielding an analogous admissible embedding ${}^LT_{H, i} \to {}^LH_i$.
    In fact, after enlarging $n$,
    these admissible embeddings descend to $W_{F_i}/I^n_i$ for any $B \leq i \leq \infty$.
     Then the failure of commutativity of the diagram
    $$
    \xymatrix{
    {}^LT_{H, i} \ar[r] \ar[d] & {}^LH_{i} \ar[d]^{\xi_{i}} \\
    {}^LT_{G, i} \ar[r] & {}^LG_{i},
    }
    $$
    where the left vertical map is induced by $\iota$, 
    defines a continuous $1$-cocycle of $W_{F_i}/I_{i}^{n}$
    in $\widehat{T_{G, i}}$
    as in \cite[(3.5)]{LanglandsShelstadTransfer}.
    Let $\mu_i \in H^1(W_{F_i}/I_{i}^{n}, \widehat{T_{G, i}})$ denote its class.
    Let $\chi_{\mu_{i}} \colon T_{G, i}(F_i) \to \mbC^\times$ be the continuous character corresponding to $\mu_i$ via Langlands duality for tori.
    Then we have 
    $
    \Delta_{III_2}(\gamma_{H_{i}},\gamma_{i})=\chi_{\mu_{i}}(\gamma_i)$.
    
    By the construction of $\chi$-data,
    it follows after enlarging $B$ that
    $\mu_i = \mu_\infty$
    under the identifications
    \[
    H^1(W_{F_i}/I_{i}^{n}, \widehat{T_{G, i}}) \simeq H^1(W_{F_\infty}/I_{\infty}^{n}, \widehat{T_{G, \infty}})
    \]
    for all $i \geq B$.
    Then
    one can conclude that
    $\chi_{\mu_{i}}(\gamma_i) = \chi_{\mu_{\infty}}(\gamma_\infty)$
    for large enough $i$; indeed, by arguing as in the proof of \cite[Proposition 3.3.3]{AubertVarmaLLCToriCloseFields}, this statement follows from the compatibility (\cite[Proposition 3.6.1]{Deligne84}) between reciprocity maps in local class field theory and families of close local fields.
\end{proof}

\begin{Lemma}\label{Lemma:local constancy of DeltaIV}
$
    \Delta_{IV}(\gamma_{H_{i}},\gamma_{i})=\Delta_{IV}(\gamma_{H_{\infty}},\gamma_{\infty})
    $
    for all $i\geq B$
after possibly enlarging $B$.
\end{Lemma}

\begin{proof}
    Recall that $\Delta_{IV}(\gamma_{H_{i}},\gamma_{i})=D_{G_{i}}(\gamma_{i})D_{H_{i}}(\gamma_{H_{i}})^{-1}$ where $D_{G_{i}}(\gamma_{i})= \left\vert \prod_{\alpha\in R_G} (\alpha(\gamma_{i})-1)\right\vert^{1/2}$
    and
    $D_{H_{i}}(\gamma_{H_{i}})=\left\vert \prod_{\alpha\in R_H}(\alpha(\gamma_{H_{i}})-1)\right\vert^{1/2}$.
    The assertion then follows from  \Cref{Lemma: elements of O}.
\end{proof}

Now \Cref{proposition: local constancy of transfer factors} follows from the previous four lemmas, and the proof of \Cref{Theorem: Standard Endo FL} is complete.

\subsection*{Acknowledgements}
We would like to thank Jingren Chi, Ulrich G\"ortz, Thomas Haines, Jochen Heinloth, Tetsushi Ito, Weimin Jiang, Wansu Kim, Teruhisa Koshikawa, Bertrand Lemaire, Daniel Li-Huerta, Sophie Morel and Griffin Wang for helpful discussions, interest in and feedback on this work.
S.\ Bartling would like to thank Vytautas Pa\v{s}k\={u}nas for a short but eventful conversation and especially Andreas Mihatsch for suggesting the collaboration on \cite{AndreasMeAFL} which naturally led to this project.
While working on this project, S.\ Bartling was part of the DFG RTG 2553
and K.\ Ito was supported by JSPS KAKENHI Grant Numbers 24K16887 and 24H00015.

\bibliographystyle{abbrvsort}
\bibliography{bibliography.bib}
\end{document}